%% file: RSVD_CUR.tex
\documentclass[final,onefignum,onetabnum]{siamart220329}

\usepackage{braket,amsfonts}

\usepackage{array}
\usepackage{arydshln}
\usepackage{blkarray}
\usepackage{multirow}
\usepackage[caption=false]{subfig}

\usepackage{pgfplots}
 \pgfplotsset{compat=newest,legend style={font=\scriptsize,row sep=-0.05cm,/tikz/every odd column/.append style={column sep=0.01cm}}} 
 \usepgfplotslibrary{groupplots}
 
\newsiamremark{example}{Example}
\newsiamremark{remark}{Remark}
\newsiamremark{experiment}{Experiment}
\newsiamremark{hypothesis}{Hypothesis}
\crefname{hypothesis}{Hypothesis}{Hypotheses}
\crefname{remark}{Remark}{Remarks}
\crefname{experiment}{Experiment}{Experiments}

\usepackage{algorithmic}

\usepackage{epstopdf}

\Crefname{ALC@unique}{line}{lines}

\usepackage{amsopn}
\DeclareMathOperator{\Range}{Range}

\usepackage{xspace}
\usepackage{bold-extra}
\usepackage[most]{tcolorbox}

\usepackage{amsmath}
\usepackage{array}
\usepackage{multirow}
\usepackage{enumitem}
\usepackage{subfig}
\usepackage{xcolor}
\usepackage{mathtools}
\usepackage{bm}
\newcommand{\norm}[1]{\lVert#1\rVert}
\newcommand{\R}{\mathbb R}
\newcommand{\N}{\mathbb N}
\newcommand{\bb}{\mathbf b}
\newcommand{\bff}{\mathbf f}
\newcommand{\bp}{\mathbf p}
\newcommand{\bss}{\mathbf s}
\newcommand{\bx}{\mathbf x}
\newcommand{\by}{\mathbf y}
\newcommand{\bz}{\mathbf z}
\newcommand{\bw}{\mathbf w}

\newcommand{\wh}{\widehat}

\DeclareMathOperator*{\argmax}{\arg\!\max}
\DeclareMathOperator{\diag}{\sf diag}
\definecolor{mycolor1}{rgb}{0.00000,1.00000,1.00000}%
\definecolor{mycolor2}{rgb}{0.8500, 0.3250, 0.0980}%
\definecolor{mycolor3}{rgb}{1.00000,0.00000,1.00000}%

\newcommand{\pg}[1]{{\color{black}#1}}

\newcommand{\smtxa}[2]{
{\mbox{\scriptsize
$\left[\!\!
\begin{array}{#1}
#2
\end{array} \!\! \right]$}}}
\newcommand{\mtxa}[2]{
\left[
\begin{array}{#1}
#2
\end{array}
\right]}

\begin{tcbverbatimwrite}{tmp_\jobname_header.tex}
\title{A Restricted SVD type CUR Decomposition \\ for Matrix Triplets\thanks{Version \today.\funding{This work has received funding from the European Union's Horizon 2020 research and innovation programme under the Marie Sklodowska-Curie grant agreement No 812912.}}}
\author{Perfect~Y.~Gidisu\thanks{Department of Mathematics and Computer Science, TU Eindhoven, The Netherlands, (\email{p.gidisu@tue.nl}, \email{m.e.hochstenbach@tue.nl}).} \and Michiel~E.~Hochstenbach\footnotemark[2]}

\headers{RSVD-CUR Decomposition for Matrix Triplets}{Gidisu and Hochstenbach}
\end{tcbverbatimwrite}
\input{tmp_\jobname_header.tex}

\ifpdf
\hypersetup{ pdftitle={A Restricted SVD type CUR Decomposition for Matrix Triplets}}
\fi

\begin{document}
\maketitle
\begin{tcbverbatimwrite}{tmp_\jobname_abstract.tex}
\begin{abstract}
We present a new restricted SVD-based CUR (RSVD-CUR) factorization for matrix triplets $(A, B, G)$ that aims to extract meaningful information by providing a low-rank approximation of the three matrices using a subset of their rows and columns. The proposed method utilizes the discrete empirical interpolation method (DEIM) to select the subset of rows and columns from the orthogonal and nonsingular matrices obtained through a restricted singular value decomposition of the matrix triplet. We explore the relationships between a DEIM type RSVD-CUR factorization, a DEIM type CUR factorization, and a DEIM type generalized CUR decomposition, and provide an error analysis that establishes the accuracy of the RSVD-CUR decomposition within a factor of the approximation error of the restricted singular value decomposition of the given matrices.

The RSVD-CUR factorization can be used in applications that require approximating one data matrix relative to two other given matrices. We discuss two such applications, namely multi-view dimension reduction and data perturbation problems where a correlated noise matrix is added to the input data matrix. Our numerical experiments demonstrate the advantages of the proposed method over the standard CUR approximation in these scenarios.
\end{abstract}

\begin{keywords}
Restricted SVD, low-rank approximation, CUR decomposition, interpolative decomposition, DEIM, subset selection, canonical correlation analysis, multi-view learning, nonwhite noise, colored noise, structured perturbation
\end{keywords}

\begin{MSCcodes}
65F55, 15A23, 15A18, 15A21, 65F15, 68W25
\end{MSCcodes}
\end{tcbverbatimwrite}
\input{tmp_\jobname_abstract.tex}

\section{Introduction}\label{sec: intro} Identifying the underlying structure of a data matrix and extracting meaningful information is a crucial problem in data analysis. Low-rank matrix approximation is one of the means to achieve this. CUR factorizations and interpolative decompositions (IDs) are appealing techniques for low-rank matrix approximations, which approximate a data matrix in terms of a subset of its columns and rows. These types of low-rank matrix factorizations have several advantages over the ones based on orthonormal bases because they inherit properties such as sparsity, nonnegativity, and interpretability of the original matrix. Various proposed algorithms in the literature seek to find a representative subset of rows and columns by exploiting the properties of the singular vectors \cite{Mahoney, Sorensen} or using a pivoted QR factorization \cite{Voronin}. Given a matrix $A\in \R^{m\times n}$ and a target rank $k$, a rank-$k$ CUR factorization approximates $A$ as (in line with \cite{Str19}, we will use the letter $M$ instead of $U$ for the middle matrix)
\begin{equation}\label{eq: cur}
\begin{array}{ccccc}
 A&\approx & C & M & R~, \\[-0.5mm]
 {\scriptstyle m\times n} & & {\scriptstyle m\times k} &{\scriptstyle k\times k}& {\scriptstyle k\times n}~
\end{array}
\end{equation}
where $C$ and $R$ (both of rank $k\le\min(m, n)$) consist of $k$ columns and rows of $A$, respectively. The middle matrix $M$ (of rank $k$) can be computed as $C^{+}\!AR^{+}$ ($C^+$ and $R^+$ denote the pseudoinverse of $C$ and $R$, respectively); in \cite{Stewart}, Stewart shows how this computation minimizes $\norm{A-CMR}$ for specified row and column indices. Here, $\norm{\cdot}$ denotes the spectral norm. To construct the factors $C$ and $R$, one can apply the discrete empirical interpolation method (DEIM) proposed in \cite{Chaturantabut} or any other appropriate index selection method (see, e.g., \cite{goreinov1997pseudo,Mahoney,Drmac} ) to the leading $k$ right and left singular vectors of $A$, to obtain the appropriate column and row indices.

In this paper, we generalize a DEIM type CUR \cite{Sorensen} method to develop a new coupled CUR factorization of a matrix triplet $(A, B, G)$ of compatible dimensions, based on the restricted singular value decomposition (RSVD). We call this factorization an RSVD based CUR (RSVD-CUR) factorization. We stress that this RSVD does not stand for randomized SVD (see, e.g., \cite{halko2011finding}). 

Over the decades, several generalizations of the singular value decomposition (SVD) corresponding to the product or quotient of two to three matrices have been proposed. The most commonly known generalization is the generalized SVD (GSVD), also referred to as the quotient SVD of a matrix pair $(A, B)$ \cite{chu2000computation}, which corresponds to the SVD of $AB^{-1}$ if $B$ is square and nonsingular. Another generalization is the RSVD of a matrix triplet $(A, B, G)$ \cite{zha1991restricted} which shows the SVD of $B^{-1}AG^{-1}$ if $B$ and $G$ are square and nonsingular. Similarly, we have proposed generalizations of an SVD-based CUR decomposition: first, a generalized CUR (GCUR) decomposition of a matrix pair $(A, B)$ in \cite{gidisu2021}; second, in this paper, an RSVD-CUR decomposition of a matrix triplet $(A, B, G)$. We emphasize that an RSVD-CUR is more general than a GCUR decomposition. One can derive a GCUR decomposition from an RSVD-CUR factorization given special choices of the matrices $B$ or $G$ (we will see this in \cref{pp3}); however, we note that the converse does not hold. Both CUR decomposition and RSVD algorithms have been well-studied. However, to the best of our knowledge, this work is the first to combine both methods. The RSVD has been around for over three decades now; this new method introduces a new type of exploitation of the RSVD. 

In recent times, real-world data sets often contain multiple representations or viewpoints, each providing unique and complementary information. Our motivation for the RSVD-CUR factorization comes from the canonical correlation analysis (CCA) of a matrix pair $(B, G)$ (see, e.g., \cite{golub1995canonical}), which is related to the RSVD of the matrix triplet $(B^T\!G, B^T\!, G)$ (see \cref{sec: RSVD}). 
 CCA is a popular method for analyzing the relationships between two sets of variables, and it has broad applications in various fields \cite[pp.~443--454]{Hardle2015}. For example, in web classification problems, a document can be represented by either the words on the page (i.e., matrix $B$) or the words in the anchor text of links pointing to it (i.e., matrix $G$). Similarly, in a genome-wide association study, CCA is used to find genetic associations between genotype data (contained in $B$) and phenotype data (contained in $G$) \cite{chen2012structured}. CCA aims to find linear combinations of variables from each data set that exhibits the highest correlation with each other. These linear combinations, represented by the canonical vectors, form a basis for the correlated subspaces of the data sets. The first canonical vector pair has the highest correlation, and subsequent pairs have decreasing correlations.   

 Our goal is to extract subsets of columns or rows from $B$ and $G$ by utilizing the canonical vectors of each matrix that maximize the correlations between them. We believe that the RSVD-CUR factorization can be useful for multi-view dimension reduction and integration of information from multiple views in multi-view learning, a rapidly growing area of machine learning that involves using multiple perspectives to improve generalization performance \cite{xu2013survey}. Similar to CCA, the RSVD-CUR factorization can handle two-view cases and may also be utilized as a supervised feature selection technique in multi-label classification problems, where one view comes from the data and the other from the class labels.

Another motivation for an RSVD-CUR factorization stems from applications where the goal is to select a subset of rows and columns of one data set relative to two other data sets. An example is a data perturbation problem of the form $A_E=A + BFG$ where $BFG$ is a correlated noise matrix (see, e.g., \cite{beck2007matrix,zha1991restricted}) and the goal is to recover the low-rank matrix $A$ from $A_E$ given the structure of $B$ and $G$. Conventionally, when faced with this kind of perturbation problem, to use an SVD-based method, a prewhitening step is required to make the additive correlated noise a white noise using $B^{-1}$ and $G^{-1}$. However, with the RSVD formulation, the prewhitening operation becomes an integral part of the algorithm. It is worth pointing out that one does not necessarily need to know the exact noise covariance matrices; the RSVD and RSVD-CUR may still deliver good approximation results given inexact covariance matrices (see \cref{exp:2}). An example of an inexact covariance matrix is when we approximate the population covariance matrix by a sample covariance matrix.

\pg {Considering the ordinary or total least squares problem of the form $A\bx\approx \mathbf b$, in many applications, it is desirable to reduce the number of variables that are to be considered or measured in the future. For instance, the modeler may not be interested in a predictor such as $A\bx$ with all redundant variables but rather $A\wh \bx$, where $\wh \bx$ has at most $k$ nonzero entries. The position of the nonzero entries determines which columns of $A$, i.e., which variables to use in the model for approximating the response vector $\mathbf b$. How to pick these columns is the problem of subset selection, and one can use a CUR factorization algorithm. Consider the setting of generalized Gauss-Markov models with constraints, i.e., 
\begin{equation}\label{eq:guass-markov}
 \min_{\bx, \by, \bff} \ \norm{\by}^2 +\norm{\bff}^2 \quad \text{subject to}\quad \mathbf b=A\bx+B\by, \quad \bff =G\bx,
\end{equation}
where $A, B, G, \mathbf b$ are given. Notice that where $B=I$ and $G=0$, this formulation is a generalization of the traditional least squares problem. Since this equation involves three matrices, an appropriate tool for its analysis will be the RSVD \cite{de1991restricted,hansen}. For variable subset selection in this problem, the RSVD-CUR may be a suitable method that incorporates the error and the constraints (more details in \cref{sec: exper}).
}  

\textbf{Outline.} 
\Cref{sec: RSVD} gives a brief overview of the RSVD. \Cref{sec: Err-RSVD-CUR} introduces the new RSVD-CUR decomposition. In this section, we also discuss some error bounds. \Cref{algo: RSVD-CUR-DEIM} summarizes the procedure of constructing a DEIM type RSVD-CUR decomposition. Results of numerical experiments using synthetic and real data sets are presented in \cref{sec: exper}, followed by conclusions in \cref{sec: con}.

\section{Restricted SVD} \label{sec: RSVD} The RSVD of matrix triplets, as notably studied in \cite{zha1991restricted,de1991restricted}, is an essential building block for the proposed decomposition in this paper. We give a brief overview of this matrix factorization here. The RSVD may be viewed as a decomposition of a matrix relative to two other matrices of compatible dimensions. Given a matrix triplet $A\in \R^{m \times n}$ (where, without loss of generality, $m\ge n$), $B \in \R^{m \times \ell}$, and $G\in \R^{d\times n}$, following the formulation in \cite{zha1991restricted}, there exist orthogonal matrices $U\in \R^{\ell \times \ell}$ and $V \in \R^{d \times d}$, and nonsingular matrices $Z \in \R^{m \times m}$ and $W \in \R^{n \times n}$ such that 
\begin{equation}\label{eq: RSVD}
A = Z \, D_A \, W^T, \quad
B = Z \, D_B \, U^T, \quad
G = V \, D_G \, W^T,
\end{equation}
which implies
\[\begin{bmatrix}A & B \\ G \end{bmatrix}= \begin{bmatrix}Z & \\ & V\end{bmatrix} \begin{bmatrix}D_A & D_B \\ D_G\end{bmatrix} \begin{bmatrix}W & \\ & U\end{bmatrix}^T,
\]
where $D_A \in \R^{m\times n}$, $D_B \in \R^{m\times \ell}$, and $D_G \in \R^{d\times n}$ are quasi-diagonal matrices \footnote{A quasi-diagonal matrix, in this work, is a matrix that is diagonal after removing all zero rows and columns.}. We refer the reader to \cite{zha1991restricted} for detailed proof of the above decomposition. In the case of $m<n$, it is logical to take the transpose of the matrix triplet and interchange the position of $B$ and $G$ to ensure compatible dimensions, i.e., $(A^T,G^T,B^T$). With respect to the theory, applications and our experiments, we focus on the so-called {\em regular} matrix triplet $(A, B, G)$, i.e., $B$ is of full row rank and $G$ is of full column rank \cite{zha1992numerical}.

Algorithms for the computation of the RSVD are still an active field of research; some recent works include \cite{chu2000computation,zwaan2020robust}. As noted in \cite{de1991restricted}, the RSVD can be computed via a double GSVD. Following the formulation of the GSVD proposed by Van Loan \cite{Van}: Given $A\in\R^{m\times n}$ and $G\in\R^{d\times n}$ with $m, d \ge n$, there exist matrices $U \in{ \mathbb R ^{m \times m}}$, $V \in{ \mathbb R^{d \times d}}$ with orthonormal columns and a nonsingular $X \in {\mathbb R ^{n \times n}}$ such that
\begin{equation}
\begin{aligned}
\label{gsvd}
U^T\!AX &= \Gamma = \text{diag}(\gamma_1,\dots,\gamma_n), \qquad &\gamma_i\in [0,1],\\
V^T\!GX &=\Sigma = \text{diag}(\sigma_1,\dots,\sigma_n), \qquad &\sigma_i\in [0,1],
\end{aligned}
\end{equation}
where $\gamma_i^{2}+\sigma_i^{2} = 1 $. Let $Y := X^{-T}$ in the GSVD of \eqref{gsvd}. Then $A = U \Gamma Y^T$ and $G = V \Sigma Y^T$. The following is a practical procedure to construct the RSVD using the GSVD. For ease of presentation, we first assume that $m = \ell$ and $d = n$ so that $B$ and $G$ are square. Then, we have the following expression as the RSVD from two GSVDs:
\begin{align*}
\begin{bmatrix}A & B \\ G\end{bmatrix}
& = \begin{bmatrix}U_1 & \\ & V_1\end{bmatrix} \begin{bmatrix}\Gamma_1 & U_1^T B \\ \Sigma_1\end{bmatrix} \begin{bmatrix}Y_1^T \\ & I\end{bmatrix} \\
& = \begin{bmatrix}U_1 & \\ & V_1\end{bmatrix} \begin{bmatrix}\Gamma_1 \Sigma_1^{-1} & U_1^T B \\ I\end{bmatrix} \begin{bmatrix}\Sigma_1 Y_1^T \\ & I\end{bmatrix} \\
& = \begin{bmatrix}U_1 Y_2 & \\ & V_1\end{bmatrix} \begin{bmatrix}\Sigma_2^T & \Gamma_2^T \\ V_2\end{bmatrix} \begin{bmatrix}V_2^T \Sigma_1 Y_1^T \\ & U_2^T\end{bmatrix} \\
& = \begin{bmatrix}U_1 Y_2 & \\ & V_1 V_2\end{bmatrix} \begin{bmatrix}\Sigma_2^T \, \Gamma_G & \Gamma_2^T \\ \Gamma_G\end{bmatrix} \begin{bmatrix}Y_1 \Sigma_1 V_2 \Gamma_G^{-1} \\ & U_2\end{bmatrix}^T.
\end{align*}
The identity matrix is denoted by $I$. In these four steps, we have first computed the GSVD of $(A, G)$, i.e., $A = U_1 \, \Gamma_1 \, Y_1^T$ and $G = V_1 \, \Sigma_1 \, Y_1^T$.
Note that $\Sigma_1$ is nonsingular since $G$ is nonsingular.
Next, we compute the GSVD of the transposes of the pair $(U_1^TB, \, \Gamma_1 \Sigma_1^{-1})$, so that $U_1^T B = Y_2 \, \Gamma_2^T \, U_2^T$ and $\Gamma_1 \Sigma_1^{-1} = Y_2 \, \Sigma_2^T \, V_2^T$.
Moreover, $\Gamma_G$ is a nonsingular scaling matrix that one can freely select (see, e.g., \cite{zwaan2020robust}).
In this square case, we have $\Sigma_2^T = \Sigma_2$, but we keep this notation for consistency with the nonsquare case we will discuss now.

In some of our applications of interest (see \cref{exp:3}), we have that $\ell = d > m \ge n$.
In this case, we get the following modifications:
\begin{align*}
& \begin{bmatrix}U_1 & \\ & V_1\end{bmatrix} \begin{bmatrix}\Gamma_1 & U_1^T B \\ \begin{bmatrix}\Sigma_1 \\ 0_{d-n,n}\end{bmatrix}\end{bmatrix} \begin{bmatrix}Y_1^T \\ & I\end{bmatrix} \\
& = \begin{bmatrix}U_1 & \\ & V_1\end{bmatrix} \begin{bmatrix}\Gamma_1 \Sigma_1^{-1} & U_1^T B \\ \begin{bmatrix}I \\ 0_{d-n,n}\end{bmatrix}\end{bmatrix} \begin{bmatrix}\Sigma_1 Y_1^T \\ & I\end{bmatrix} \\
& = \begin{bmatrix}U_1 Y_2 & \\ & V_1\end{bmatrix} \begin{bmatrix}\Sigma_2^T & \Gamma_2^T \\ \begin{bmatrix}V_2 \\[-1mm] 0_{d-n,n}\end{bmatrix}\end{bmatrix} \begin{bmatrix}V_2^T \Sigma_1 Y_1^T \\ & U_2^T\end{bmatrix} \\
& = \begin{bmatrix}U_1 Y_2 & \\ & V_1 \wh V_2\end{bmatrix} \begin{bmatrix}\Sigma_2^T \, \Gamma_G & \Gamma_2^T\\ \begin{bmatrix}\Gamma_G \\[-1mm] 0_{d-n,n}\end{bmatrix}\end{bmatrix} \begin{bmatrix}Y_1 \Sigma_1 V_2 \Gamma_G^{-1} \\ & U_2 \end{bmatrix}^T.
\end{align*} 
In these steps, we use $\wh V_2 = \diag(V_2, I_{d-n})$. \cref{algo: RSVD} summarizes the procedure for computing the RSVD of the so-called {\em regular} matrix triplets $(A, B, G)$.
\begin{algorithm}[htb!]
{\footnotesize\begin{algorithmic}[1]
\REQUIRE $A\in \R^{m\times n}$, $B\in \R^{m\times \ell}$, $G\in \R^{d\times n}$, \quad $m \ge n$, \ $m \le \ell$, \ $d \ge n$
\ENSURE $Z\in\R^{m\times m}$, $W\in \R^{n\times n}$, $U\in\R^{\ell\times\ell}$, $V\in\R^{d\times d}$ \\\hspace{7mm} $D_A \in \R^{m\times n}$, $D_B \in \R^{m\times \ell}$, and $D_G \in \R^{d\times n}$  \quad (see \eqref{eq: RSVD})
\STATE Compute  $[U_1,V_1,Y_1,\Gamma_1,\Sigma_1] = {\sf gsvd}(A, G)$
\STATE Set $\Sigma_1 = \Sigma_1(1:n,:)$
\STATE Compute $[U_2, V_2,Y_2,\Gamma_2,\Sigma_2] = {\sf gsvd}(B^TU_1, \, (\Gamma_1\Sigma_1^{-1})^T)$
\STATE $\mathbf a = {\sf diag}(\Sigma_2) \quad (\in \R^n)$  
\STATE $\Gamma_G={\sf diag}(a_i \, (a_i^2+1)^{-1/2})$, \quad ($i = 1, \dots, n$)
\STATE {\bf if} $d = n$, \ \ $D_G = \Gamma_G;\ \  V = V_1V_2$; \ {\bf end}
\STATE {\bf if} $d > n$, \ \ $V = V_1 \cdot \text{diag}(V_2, \, I_{d-n}); \ \ D_G = [\Gamma_G; \ 0_{d-n,n}];$ \ {\bf end}
\STATE $Z=U_1Y_2$; \ \ $W= Y_1 \Sigma_1 V_2 \Gamma_G^{-1}$; \ \  $U= U_2$; \ \ $D_A = \Sigma_2^T \, \Gamma_G$; \ \ $D_B=\Gamma_2^T$
\end{algorithmic}}
 \caption{RSVD via a double GSVD}\label{algo: RSVD}
\end{algorithm}
In the two GSVD steps, we emphasize that we maintain the traditional nondecreasing ordering of the generalized singular values in both GSVDs. That is, the diagonal entries of $\Gamma_1$ and $\Gamma_2$ are in nondecreasing order while those of $\Sigma_1$ and $\Sigma_2$ are in nonincreasing order.
Note that $\Sigma_1$ is again nonsingular because $G$ is of full rank. With reference to \eqref{eq: RSVD}, we define $Z:= U_1Y_2$, $W:= Y_1 \Sigma_1 V_2 \Gamma_G^{-1}$, $V := V_1 \wh V_2$, $U:= U_2$, $D_A := \Sigma_2^T \, \Gamma_G$, $D_B:=\Gamma_2^T$, and $D_G :=\scriptsize{\begin{bmatrix}\Gamma_G \\ 0_{d-n,n}\end{bmatrix}}$. The quasi-diagonal matrices $D_A$ and $D_B$ have the following structure:
\begin{equation}\label{eq:diagmat}
D_A=\begin{bmatrix}
D_1\\
0_{m-n,n}
\end{bmatrix} \quad \text{and} \quad
 D_B=\begin{bmatrix}
 D_2& 0_{n,m-n}& 0_{n,\ell-m}\\
 0_{m-n,n} & I_{m-n}&0_{m-n,\ell-m}
 \end{bmatrix}.
\end{equation}
Write {\sf diag}$(D_1)=(\alpha_1,\dots, \alpha_n)$, {\sf diag}$(D_2)=(\beta_1,\dots, \beta_n)$, {\sf diag}$(\Gamma_G)=(\gamma_1,\dots, \gamma_n)$, and $\Sigma_2=\diag(\sigma_1,\dots,\sigma_n)$, for $i=1,\dots,n$. Note that, in view of the assumption that $B$ and $G$ are of full rank and $m\ge n$, $1>\alpha_i\ge \alpha_{i+1}>0$, $1>\gamma_i\ge \gamma_{i+1}>0$, and $0<\beta_i\le \beta_{i+1}<1$. 
As mentioned earlier, $\Gamma_G$ is a scaling matrix one can freely choose. Given $\Sigma_2$, we may, for instance, choose $\gamma_i= \sigma_i \, (\sigma_i^2+1)^{-1/2}$, which are nonzero and ordered nonincreasingly (since the function $t \mapsto t \, (t^2+1)^{-1/2}$ is strictly increasing). This implies that $\alpha_i= \sigma_i^2 \, (\sigma_i^2+1)^{-1/2}$. Given that $\beta_i^2+\sigma_i^2=1$ from the second GSVD, we have that $\alpha_i^2+\beta_i^2+\gamma^2_i=1$ for $i=1,\dots,n$. Note that $\frac{\alpha_i}{\beta_i\gamma_i}\ge\frac{\alpha_{i+1}}{\beta_{i+1}\gamma_{i+1}}$, which follows from the fact that $\alpha_i/\gamma_i=\sigma_i$, which are nonincreasing.

We now state a connection of the RSVD with CCA, which is a motivation for our proposed decomposition. In \cite{de1991restricted}, De Moor and Golub show a relation of the RSVD to a generalized eigenvalue problem. The related generalized eigenvalue problem of the RSVD of the matrix triplet $(B^T\!G, B^T\!, G)$ with $m=d$ as shown in \cite[Sec.~2.2]{de1991restricted} is 
\begin{align*}
 \begin{bmatrix}& B^T\!G\\G^TB & \end{bmatrix}\begin{bmatrix}\bz\\\bw\end{bmatrix}=\lambda\begin{bmatrix}B^T\!B & \\& G^T\!G\end{bmatrix}\begin{bmatrix}\bz\\\bw\end{bmatrix}.
\end{align*}
The above problem is exactly the generalized eigenvalue problem of the {\sf cca}$(B, G)$ (see, e.g., \cite{golub1995canonical}). Note that matrices $B^T\!B$ and $G^T\!G$ can be interpreted as covariance matrices. In applications where these covariance matrices are (nearly) singular, one may use the RSVD instead to find a solution without explicitly solving the generalized eigenvalue problem.

\section{A Restricted SVD based CUR decomposition and its approximation properties}\label{sec: Err-RSVD-CUR}
In this section, we describe the proposed RSVD-CUR decomposition and provide theoretical bounds on its approximation errors. We use MATLAB type notations to index vectors and matrices, i.e., $A(:,\bp)$ denotes the $k$ columns of $A$ with corresponding indices in vector $\bp \in \N_+^k$.

\subsection{A Restricted SVD based CUR decomposition}\label{sec: RSVD-CUR}
We now introduce a new RSVD-CUR decomposition of a matrix triplet $(A, B, G)$ with $A\in\R^{m \times n}$ (where, without loss of generality, $m\ge n$), $B\in\R^{m \times \ell}$, and $G\in\R^{d \times n}$ where $B$ and $G$ are of full rank. This RSVD-CUR factorization is guided by the knowledge of the RSVD for matrix triplets reviewed in \cref{sec: RSVD}. We now define a rank-$k$ RSVD-CUR approximation; cf.~\eqref{eq: cur}.
\begin{definition} \label{Dfn1}
Let $A$ be $m \times n$, $B$ be $m \times \ell$, and $G$ be $d \times n$. A rank-$k$ RSVD-CUR approximation of $(A, B, G)$ is defined as 
\begin{equation}\label{eq: rsvd-cur}
\begin{array}{lll}
 A &\approx& A_k:= C_A\ M_A\ R_A := AP\ M_A\ S^T\!A ~ ,\\[1mm]
 B &\approx& B_k:= C_B\ M_B\ R_B := BP_B\ M_B\ S^T\!B ~ ,\\[1mm]
 G &\approx &G_k:= C_G\ M_G \ R_G := GP \ M_G \ S_G^TG.
\end{array}
\end{equation}
Here $S \in \R^{m \times k}$, $S_G \in \R^{d \times k}$, $P \in \R^{n \times k}$, and $P_B \in \R^{\ell \times k}$ $(k \le \min(m, n, d,\ell))$ are index selection matrices with some columns of the identity that select rows and columns of the respective matrices. 
\end{definition}
It is key that {\em the same} rows of $A$ and $B$ are picked and {\em the same} columns of $A$ and $G$ are selected;
this gives a coupling among the decompositions. The matrices $C_A \in \R^{m\times k}$, $C_B\in \R^{m\times k}$, $C_G\in \R^{d\times k}$, and $R_A\in \R^{k\times n}$, $R_B\in \R^{k\times \ell}$, $R_G\in \R^{k\times n}$ are subsets of the columns and rows, respectively, of the given matrices. Let the vectors $\bss$, $\bss_G$, $\bp$, and $\bp_B$ contain the indices of the selected rows and columns so that $S=I(:,\bss)$, $S_G=I(:,\bss_G)$, $P=I(:,\bp)$, and $P_B=I(:,\bp_B)$. The choice of $\bss$, $\bss_G$, $\bp$, and $\bp_B$ is guided by the knowledge of the orthogonal and nonsingular matrices from the rank-$k$ RSVD. Given the column and row index vectors, following \cite{Sorensen,Mahoney, Stewart}, we compute the middle matrices as mentioned in \cref{sec: intro}, that is, $M_A=(C_A^T C_A)^{-1}C_A^T A R_A^T(R_A R_A^T )^{-1}$, and similarly for $M_B$ and $M_G$. 
There are several index selection strategies proposed in the literature for finding the ``best" row and column indices. The approaches we employ are the DEIM \cite{Chaturantabut} and the QDEIM \cite{Drmac} algorithms, which are greedy deterministic procedures and simple to implement.

The DEIM procedure has first been introduced in the context of model reduction of nonlinear dynamical systems \cite{Chaturantabut}; it is a discrete variant of empirical interpolation proposed in \cite{Barrault}. Sorensen and Embree later show how the DEIM algorithm is a viable column and row index selection procedure for constructing a CUR factorization \cite{Sorensen}. To construct $C$ and $R$, apply the DEIM scheme to the top $k$ right and left singular vectors, respectively \cite{Sorensen}. The DEIM procedure uses a locally optimal projection technique similar to the pivoting strategy of the LU factorization. The column and row indices are selected by processing the singular vectors sequentially as summarized in \cref{algo: DEIM} \footnote{The backslash operator used in the algorithms is a Matlab type notation for solving linear systems and least-squares problems.}. 
\begin{algorithm}[htb!]
{\footnotesize\begin{algorithmic}[1]
\REQUIRE $U \in \R^{m \times k}$ with $k\le m$ (full rank)
\ENSURE Indices $\bss \in \N_+^k$ with non-repeating entries 
\STATE $\bss(1)$ = $\argmax_{1\le i\le m}~ |(U(:,\,1))_i|$
\FOR{$j = 2, \dots, k$}
\STATE $U(:,\,j) = U(:,\,j)-U(:,\,1:j-1)\cdot (U(\bss,\,1:j-1)
\ \backslash \ U(\bss,\,j))$ 
\STATE $\bss(j)$ = $\argmax_{1\le i\le m}~ |(U(:,\,j))_i|$\hspace{3mm}
\ENDFOR
\end{algorithmic}}
 \caption{Discrete empirical interpolation index selection method ({\sf deim}) \cite{Chaturantabut}}\label{algo: DEIM}
\end{algorithm}
In \cite{Drmac}, the authors propose a QR-factorization based DEIM called QDEIM, which is much simpler than the original DEIM and enjoys a sharper error bound for the DEIM projection error. The availability of the pivoted QR implementation in many open-source packages makes this algorithm an efficient alternative index selection scheme. In the QDEIM approach, one applies a column-pivoted QR procedure on the transpose of the leading $k$ right and left singular vectors to find the indices for constructing the factors $C$ and $R$ in CUR approximation.

A DEIM type CUR decomposition requires singular vectors or approximate singular vectors. In this paper, we apply the DEIM procedure or its variant QDEIM to the nonsingular and orthogonal matrices from the RSVD instead. In an SVD-based CUR factorization, the left and right singular vectors serve as bases for the column and row spaces of matrix $A$, respectively. In our new context, the columns of matrices $Z$ and $W$ from \cref{eq: RSVD} may be viewed as bases for the column and row spaces, respectively, of $A$ relative to the column space of $B$ and the row space of $G$. The procedure for constructing a DEIM type RSVD-CUR is summarized in \cref{algo: RSVD-CUR-DEIM}. 
\begin{algorithm}[htb!]
{\footnotesize\begin{algorithmic}[1]
\setcounter{ALC@unique}{0}
 \REQUIRE $A \in \R^{m \times n}$, $B \in \R^{m \times \ell}$, $G \in \R^{d \times n}$, desired rank $k$ 
 \ENSURE A rank-$k$ RSVD-CUR decomposition \\
$A_k = A(:,\bp) \, \cdot \, M_A \, \cdot \, A(\bss,:)$, \ 
$B_k = B(:,\bp_B) \, \cdot \, M_B \, \cdot \, B(\bss,:)$, \ 
$G_k = G(:,\bp) \, \cdot \, M_G \, \cdot \, G(\bss_G,:)$

 \STATE Compute rank-$k$ RSVD of $(A,B,G)$ to obtain $W, Z, U , V$ \hfill (see \eqref{eq: RSVD}) \label{line1}
	\STATE\label{line2}{$\bp = {\sf deim}(W)$ \hfill (perform DEIM on the matrices from the RSVD)}
	\STATE\label{line3}{$\bss = {\sf deim}(Z)$}
 \STATE\label{line6}{$M_A = A(\,:,\bp) \, \backslash \ (A \, / \, A(\bss,:\,))$}
	\STATE\label{line4}{$\bp_B = {\sf deim}(U)$} \hfill (optional)
	\STATE\label{line5}{$\bss_G = {\sf deim}(V)$} \hfill (optional)
  \STATE\label{line7}Compute $M_B$ and $M_G$ as in \cref{line6} if needed
\end{algorithmic}}
\caption{DEIM type RSVD-CUR decomposition}\label{algo: RSVD-CUR-DEIM}
\end{algorithm}
In this implementation, the user is supposed to specify $k$. We note that one can also determine $k$ by comparing the decaying restricted singular values  $\frac{\alpha_i}{\beta_i\gamma_i}$ against a given
threshold. In some applications, the explicit approximation of $B$ or $G$ may not be necessary. Thus, \cref{line4,line5,line7} in \cref{algo: RSVD-CUR-DEIM} should only be implemented if necessary.
\begin{remark}
In \cref{line1} of \cref{algo: RSVD-CUR-DEIM}, the columns of $W$, $Z$, $U$, and $V$ corresponds to the $k$ largest restricted singular values $\frac{\alpha_i}{\beta_i\gamma_i}$. This implies that we select the most ``dominant" parts of $A$ and the least ``dominant" parts of $B$ and $G$, so in \eqref{eq: rsvd-cur} the relative approximation error of $A$ tends to be modest, while this may not be the case for the relative approximation errors of $B$ and $G$. 
\end{remark}

In many applications, as we will see in \cref{sec: exper}, one is interested in selecting only the key columns or rows and not the explicit $A\approx C_A M_A R_A$ factorization. An interpolative decomposition aims to identify a set of skeleton columns or rows of a matrix. A CUR factorization may be viewed as evaluating the ID for both the column and row spaces of a matrix simultaneously. The following are the column and row versions of an RSVD-ID factorization of a matrix triplet:
\begin{equation}\label{RSVD:ID}
 \begin{aligned}
&A \approx C_A\widetilde M_A, \quad B \approx C_B\widetilde M_B, \quad G \approx C_G\widetilde M_G, \quad \text{or} \\
&A\approx \wh M_A R_A, \quad B\approx \wh M_B R_B, \quad G\approx \wh M_G R_G.
\end{aligned} 
\end{equation}
Here, $\widetilde M_A=C_A^+\!A$ is $k \times n$ and $\wh M_A=AR^+_A$ is $m \times k$; analogous remarks hold for $\widetilde M_B$, $\widetilde M_G$, $\wh M_B$, and $\wh M_G$. Notice that in \cref{algo: RSVD-CUR-DEIM}, the key column and row indices of the various matrices are picked independently. This algorithm can therefore be restricted to select only column indices if we are interested in the column version of the RSVD-ID factorization or select only row indices if we are interested in the row version. 

De Moor and Golub \cite{de1991restricted} show the relation between the RSVD and the SVD and its other generalizations. We indicate in the following proposition the corresponding connection between the DEIM type RSVD-CUR and the (generalized) CUR decomposition \cite{Sorensen,gidisu2021}. 

\begin{proposition}\label{pp3} (i) If $B$ and $G$ are nonsingular matrices, then the selected row and column indices from a CUR decomposition of $B^{-1}AG^{-1}$ are the same as index vectors $\bp_B$ and $\bss_G$, respectively, obtained from an RSVD-CUR decomposition of $(A, B, G)$.

(ii) Furthermore, in the particular case where $B=I$ and $G=I$, the RSVD-CUR decomposition of $A$ coincides with a CUR decomposition of $A$, in that the factors $C$ and $R$ of $A$ are the same for both methods: the first line of \eqref{eq: rsvd-cur} is equal to \eqref{eq: cur}.

(iii) Lastly, given a special choice of $B=I$, an RSVD-CUR decomposition of $A$ and $G$ coincides with the GCUR decomposition of $(A, G)$ (see \cite[Def.~4.1]{gidisu2021}), in that the factors $C_A, C_G$ and $R_A, R_G$ of $A$ and $G$ are the same for both methods. In the dual case that $G=I$, similar remarks hold. 
\end{proposition}

\begin{proof} (i) We start with the RSVD \eqref{eq: RSVD}. If $B$ and $G$ are nonsingular, then the SVD of $B^{-1}AG^{-1}$ can be expressed in terms of the RSVD of $(A, B, G)$, and is equal to $U (D_B^{-1}D_AD_G^{-1}) V^T$ given that $B^{-1} = U D_B^{-1} Z^{-1}$ and $G^{-1}=W^{-T}D_G^{-1}V^T$ \cite{de1991restricted}. Consequently, the row and column index vectors from a CUR factorization of $B^{-1}AG^{-1}$ are equal to the vectors $\bss_G$ and $\bp_B$, respectively, from an RSVD-CUR of $(A, B, G)$ since they are obtained by applying DEIM to matrices $V$ and $U$, respectively.

(ii) If $B=I$ and $G=I$, from \eqref{eq: RSVD}, $I=ZD_BU^T$ and $I=VD_GW^T$ which implies $UD_B^{-1}=Z$ and $W^T=D_G^{-1}V^T$. Hence, we find that $A=UD_B^{-1}D_AD_G^{-1}V^T$ which is an SVD of $A$. Therefore the selection matrices $P$, $S$ from CUR of $A$ \eqref{eq: cur} are equal to the selection matrices $P=P_B$, $S=S_G$ from an RSVD-CUR of $(A, I, I)$ \eqref{eq: rsvd-cur}.

(iii) If $B=I$, again from \eqref{eq: RSVD}, $I=ZD_BU^T$, which implies $UD_B^{-1}=Z$. Then $A=UD_B^{-1}D_AW^T$, $G=VD_GW^T$ which is (up to a diagonal scaling) the GSVD of the matrix pair $(A, G)$; see \eqref{gsvd} \cite{de1991restricted}. Thus, the column and row selection matrices from GCUR of $(A, G)$ (see \cite[Def.~4.1]{gidisu2021}) are the same as the column and row selection matrices $P$, $S$, $S_G$ from \eqref{eq: rsvd-cur}, respectively.
\end{proof}
\subsection{Error Analysis} We begin by analyzing the error of a rank-$k$ RSVD of a matrix triplet $A\in \R^{m \times n}$ (where without loss of generality $m\ge n$), $B \in \R^{m \times \ell}$, and $G\in \R^{d\times n}$ (where $B$ and $G$ are of full rank). Given our applications of interest in \cref{sec: exper}, we consider the case $\ell = d \ge m \ge n$. To define a rank-$k$ RSVD, let us partition the following matrices
\begin{align*}\label{eq: prsvd}
U &= [U_k \ \, \widehat{U}]~, \ \, V = [V_k \ \, \widehat{V}]~,\ \, W = [W_k \ \, \widehat{W}]~, \ \, Z = [Z_k \ \, \widehat{Z}]~,\\[1mm]
 D_A &= {\sf diag}( D_{A_k},\,\widehat D_A) ~,\ \, D_B = {\sf diag}(D_{B_k},\, \widehat D_B)~, \ \, D_G = {\sf diag}( D_{G_k}, \, \widehat D_G),
\end{align*}
where $\wh D_A \in \R^{(m-k)\times (n-k)}$, $\wh D_B \in \R^{(m-k)\times (\ell-k)}$, and $\wh D_G \in \R^{(d-k)\times (n-k)}$.
We define a rank-$k$ RSVD of $(A, B, G)$ as 
\begin{equation}
 \label{eq: trsvd}
 A_k := Z_k D_{A_k} W_k^T, \qquad B_k := Z_k D_{B_k} U_k^T, \qquad G_k := V_k D_{G_k} W_k^T,
\end{equation}
where $k < n$. It follows that 
\begin{equation} \label{eq: ersvd} A - A_k = \widehat{Z} \, \widehat{D}_A \, \widehat{W}^T, \qquad B - B_k = \widehat{Z} \, \widehat{D}_B \, \widehat{U}^T, \qquad G - G_k = \widehat{V} \, \widehat{D}_G \, \widehat{W}^T. \end{equation}
The following statements are a stepping stone for the error bound analysis of an RSVD-CUR. Denote the $i$th singular value of $A$ by $\psi_i(A)$. Let $A - A_k = \widehat{Z} \, \widehat{D}_A \, \widehat{W}^T$ as in \eqref{eq: ersvd}. Then for $i=1,\dots,n$, $\psi_i(\widehat{Z} \, \widehat{D}_A \, \widehat{W}^T) \leq \psi_i(\widehat{D}_A) ~ \norm{\wh Z}\,\norm{\wh W}$ (see, e.g., \cite[p.~346]{Horn}). Since the diagonal elements of $\widehat{D}_A$ are in nonincreasing order, we have 
$\norm {A - A_k} \leq \psi_1({\wh D}_A) ~ \norm{\wh Z}\,\norm{\wh W}\leq {\alpha_{k+1}}\cdot \norm{\widehat{Z}}\,\norm{\widehat{W}}$.

Similarly, we have that $ \norm {B - B_k}=\norm{\widehat{Z} \, \widehat{D}_B \, \widehat{U}^T} \leq \norm{\widehat{Z}}$ and $\norm {G - G_k}=\norm{\widehat{V} \, \widehat{D}_G \, \widehat{W}^T} \leq {\gamma_{k+1}}\cdot\norm{\widehat{W}}$. The first inequality follows from the fact $\wh U$ has orthonormal columns and the diagonal elements of $\widehat{D}_B$ are in nondecreasing order with a maximum value of 1, so we have that $\psi_1(\wh {D}_B)=1$ (see \cref{eq:diagmat} for the structure of $D_B$) and $\norm {\wh U}=1$. The second equality is a result of the fact that $\wh V$ has orthonormal columns and the diagonal entries of $\widehat{D}_G$ are in nonincreasing order, therefore, $\psi_1(\wh {D}_G)=\gamma_{k+1}$ and $\norm {\wh V}=1$.

We now introduce some error bounds of an RSVD-CUR decomposition in terms of the error of a rank-$k$ RSVD. The analysis closely follows the error bound analysis in \cite{Sorensen, gidisu2021} for the DEIM type CUR and DEIM type GCUR methods with some necessary modifications. As with the DEIM type GCUR method, here also, the lack of orthogonality of the vectors in $W$ and $Z$ from the RSVD necessitates some additional work. We take QR factorizations of $W$ and $Z$ to obtain their respective orthonormal bases to facilitate the analysis, introducing terms in the error bound associated with the triangular matrix in the QR factorizations. 

For the analysis, we use the following QR decompositions of the nonsingular matrices from the RSVD (see \cref{eq: RSVD})
\begin{align} \label{eq: QR}
\begin{split}
[Z_k \ \ \wh Z] &= Z = Q_ZT_Z = [Q_{Z_k} \ \ {\wh Q}_Z]
\begin{bmatrix}T_{Z_k} & T_{Z_{12}} \\ 0 & T_{Z_{22}} \end{bmatrix} = [Q_{Z_k} T_{Z_k} \ \ Q_Z {\wh T}_Z],\\
[W_k \ \ \wh W] &= W = Q_WT_W = [Q_{W_k} \ \ {\wh Q}_W]
\begin{bmatrix} T_{W_k} & T_{W_{12}} \\ 0 & T_{W_{22}} \end{bmatrix} = [Q_{W_k} T_{W_k} \ \ Q_W {\wh T}_W],
\end{split}
\end{align}
where we have defined
\begin{equation} \label{eq: T_hat}
{\wh T}_Z := \begin{bmatrix} T_{Z_{12}} \\ T_{Z_{22}} \end{bmatrix}~,\qquad {\wh T}_W := \begin{bmatrix} T_{W_{12}} \\ T_{W_{22}} \end{bmatrix}.
\end{equation}
This implies that
\begin{equation}\label{part:qr}
\begin{aligned}
A &= A_k + \wh Z \, \wh D_A \, \wh W^T = Z_k D_{A_k} W_k^T + \wh Z \, {\wh D}_A \, \wh W^T\\
&= Q_{Z_k} T_{Z_k} D_{A_k} T_{W_k}^T Q_{W_k}^T + Q_Z{\wh T}_Z \, {\wh D}_A \, {\wh T}_W^T Q_W^T~,\\
B &= B_k + \wh Z \, \wh D_B \, \wh U^T = Z_k D_{B_k} U_k^T + \wh Z \, {\wh D}_B \, \wh U^T=Q_{Z_k} T_{Z_k} D_{B_k} U_k^T + Q_Z{\wh T}_Z \, {\wh D}_B \, {\wh U}^T~, \\
G &= G_k + \wh V \, \wh D_G \, \wh W^T = V_k D_{G_k} W_k^T + \wh V \, {\wh D}_G \, \wh W^T = V_k D_{G_k} T_{W_k}^T Q_{W_k}^T + \wh V \, {\wh D}_G \, {\wh T}_W^T Q_W^T.
\end{aligned}
\end{equation}
Given an orthonormal matrix $Q_W\in \R^{n\times k}$, from \cite{Sorensen,gidisu2021} as well as here, we have that the quantity $\norm{A(I-Q_{W_k}Q_{W_k}^T)}$ is key in the error bound analysis. Here, we have that $\norm{A(I-Q_{W_k}Q_{W_k}^T)}$ may not be close to $\psi_{k+1}(A)$ since the matrix $Q_{W_k}$ is from the RSVD, therefore we provide a bound on this quantity in terms of the error in the RSVD. 

Let $P$ be an index selection matrix derived from performing the DEIM scheme on matrix $W_k$. Suppose $Q_{W_k}$ is an orthonormal basis for $\Range(W_k)$, with $W_k^T\!P$ and $Q_{W_k}^T\!P$ being nonsingular, we have the interpolatory projector $P(W_k^T\!P)^{-1}W_k^T = P(Q_{W_k}^T\!P)^{-1}Q_{W_k}^T$ (see \cite[Def.~3.1, (3.6)]{Chaturantabut}). With this equality, we exploit the special properties of an orthogonal matrix by using the orthonormal bases of the nonsingular matrices from the RSVD instead for our analysis. Let $Q_{W_k}^T\!P$ and $S^TQ_{Z_k}$ be nonsingular so that $\mathbb{P} = P(Q_{W_k}^T\!P)^{-1}Q_{W_k}^T$ and $\mathbb{S}=Q_{Z_k}(S^TQ_{Z_k})^{-1}S^T$ are oblique projectors. 

\pg{In the following theorem, we provide bounds on the coupled CUR decompositions of $A$, $B$, and $G$ in terms of the RSVD quantities.
The upper bounds contain both multiplicative factors (the $\eta$'s) and the $\alpha_{k+1}$, $\gamma_{k+1}$ (both bounded by 1), and $T$-quantities, which are from the error of the truncated RSVD as defined in \cref{eq: trsvd,eq: ersvd}. }

\begin{theorem} \label{theorem1} (Generalization of \cite[Thm.~4.1]{Sorensen} and \cite[Thm.~4.8]{gidisu2021}) Given $A$, $B$, and $G$ as in \cref{Dfn1} and $Z_k\in \R^{m\times k}$, $W_k\in \R^{n\times k}$, $U_k \in \R^{\ell \times k}$, and $V_k \in \R^{d\times k}$ from \eqref{eq: trsvd}, let $Q_{Z_k} \in \R^{m\times k}$, $Q_{W_k} \in \R^{n\times k}$ be the $Q$-factors of $Z_k, W_k$, respectively, and $\widehat{T}_Z$, $T_{Z_{22}}$, $\widehat{T}_W$, and $T_{W_{22}}$ as in \eqref{eq: QR}--\eqref{eq: T_hat}. Suppose $Q_{W_k}^T\!P$, $U_k^T\!P_B$, $S_G^TV_k$, and $S^TQ_{Z_k}$ are nonsingular, then with the error constants 
\[\eta_p := \norm{(Q_{W_k}^T\!P)^{-1}}, \ \eta_s := \norm{(S^TQ_{Z_k})^{-1}},\ \eta_{p_{B}} := \norm{(U_k^T\!P_B)^{-1}}, \ \eta_{s_{G}} := \norm{(S_G^TV_k)^{-1}},\]
we have
\begin{equation}\label{eq: error}
\begin{aligned}
\norm{A-C_AM_AR_A} &\leq \alpha_{k+1} \cdot (\eta_p\cdot\norm{{\wh T}_Z}\,\norm{T_{W_{22}}} +\eta_s\cdot\norm{ T_{Z_{22}}}\,\norm{{\wh T}_W}) \\
&\le \alpha_{k+1} \cdot (\eta_p + \eta_s) \cdot \norm{{\wh T}_W}\,\norm{{\wh T}_Z}~, \\ 
\norm{B-C_BM_BR_B} &\leq \eta_{p_{B}}\cdot\norm{T_{Z_{22}}} +\eta_s\cdot\norm{{\wh T}_Z} \le (\eta_{p_{B}} + \eta_s) \cdot \norm{{\wh T}_Z}~, \\ 
\norm{G-C_GM_GR_G} &\leq \gamma_{k+1} \cdot (\eta_p\cdot\norm{{\wh T}_W} +\eta_{s_{G}}\cdot\norm{T_{W_{22}}}) \le \gamma_{k+1} \cdot (\eta_p + \eta_{s_{G}}) \cdot \norm{{\wh T}_W}. 
\end{aligned}
\end{equation}
\end{theorem}

\begin{proof}
We will prove the result for $\norm{A-C_AM_AR_A}$; results for $\norm{B-C_BM_BR_B}$ and $\norm{G-C_GM_GR_G}$ follow similarly. We first show the bounds on the errors between $A$ and its interpolatory projections $\mathbb{P}A$ and $A\mathbb{S}$, i.e., the selected rows and columns. Then, using the fact that these bounds also apply to the orthogonal projections of $A$ onto the same column and row spaces \cite[Lemma~4.2]{Sorensen}, we prove the bound on the approximation of $A$ by an RSVD-CUR.

Given the projector $\mathbb{P}$, we have that $Q_{W_k}^T\!\mathbb{P}=Q_{W_k}^T\!P(Q_{W_k}^TP)^{-1}Q_{W_k}^T=Q_{W_k}^T$, which implies $Q_{W_k}^T\!(I-\mathbb{P}) = 0$. Therefore the error in the oblique projection of $A$ is (cf. \cite[Lemma~4.1]{Sorensen})
\begin{align*}\norm{A-A\mathbb{P}}&=\norm{A(I-\mathbb{P})} = \norm{A(I - Q_{W_k}\!Q_{W_k}^T)(I-\mathbb{P})}\\
&\leq \norm{A(I - Q_{W_k}\!Q_{W_k}^T)}~\norm{I-\mathbb{P}}.\end{align*}
Note that, since $k<n$, $\mathbb{P} \ne 0$ and $\mathbb{P} \ne I$, it is well known that (see, e.g., \cite{Daniel})
\[\norm{I-\mathbb{P}} =\norm{\mathbb{P}} =\norm{(Q_{W_k}^T\!P)^{-1}}.\]
Using the partitioning of $A$ in \eqref{part:qr}, we have
\begin{align*}
A\ Q_{W_k}\, Q_{W_k}^T &= [Q_{Z_k} \ \ {\wh Q}_Z]
\begin{bmatrix} T_{Z_k} & T_{Z_{12}} \\ 0 & T_{Z_{22}} \end{bmatrix} \begin{bmatrix} D_{A_k} & 0\\ 0 & \widehat D_A \end{bmatrix} \begin{bmatrix} T_{W_k}^T & 0 \\[0.5mm] T_{W_{12}}^T & T_{W_{22}}^T \end{bmatrix} \begin{bmatrix} I_k \\ 0 \end{bmatrix} Q_{W_k}^T \\
&= Q_{Z_k} T_{Z_k} D_{A_k} T_{W_k}^T Q_{W_k}^T + Q_Z {\wh T}_Z \widehat D_A \, T_{W_{12}}^TQ_{W_k}^T,
\end{align*}
and hence
\begin{align*}
A\,(I-Q_{W_k}Q_{W_k}^T) & = (A-A_k)- Q_Z {\wh T}_Z \widehat D_A \, T_{W_{12}}^TQ_{W_k}^T \\
& = Q_Z{\wh T}_Z {\wh D}_A \, {\wh T}_W^T Q_W^T - Q_Z {\wh T}_Z \widehat D_A \, T_{W_{12}}^TQ_{W_k}^T = Q_Z{\wh T}_Z \wh D_A \, T_{W_{22}}^T {\wh Q^T}_W.
\end{align*}
This implies
\[\norm{A\,(I-Q_{W_k}Q_{W_k}^T)} \le \norm{\wh D_A} \,\norm{{\wh T}_Z}\, \norm{T_{W_{22}}} \le \alpha_{k+1} \cdot\norm{{\wh T}_Z}\, \norm{T_{W_{22}}},\]
and 
\[\norm{A(I-\mathbb{P})}\le\alpha_{k+1} \cdot\norm{(Q_{W_k}^T\!P)^{-1}}\,\norm{{\wh T}_Z}\, \norm{T_{W_{22}}}.\]
Let us now consider the operation on the left side of $A$. Given that $S^TQ_{Z_k}$ is nonsingular, we have the DEIM interpolatory projector $\mathbb{S}=Q_{Z_k}(S^TQ_{Z_k})^{-1}S^T$. It is known that (see \cite[Lemma~4.1]{Sorensen})
\begin{align*}\norm{A-\mathbb{S}A} &=\norm{(I-\mathbb{S})A}=\norm{(I-\mathbb{S})(I-Q_{Z_k}Q_{Z_k}^T)A}\\
&\le \norm{(I-\mathbb{S})}\,\norm{(I-Q_{Z_k}Q_{Z_k}^T)A}.
\end{align*}
Similar to before, since $k<m$, we know that $\mathbb{S} \ne 0$ and $\mathbb{S} \ne I$ hence 
\[\norm{I-\mathbb{S}} =\norm{\mathbb{S}} =\norm{(S^TQ_{Z_k})^{-1}}.\] 
In the same setting as earlier, we have the following expansion 
\begin{align*}
 Q_{Z_k}Q_{Z_k}^TA&=[Q_{Z_k} \ \ 0]
\begin{bmatrix} T_{Z_k} & T_{Z_{12}} \\ 0 & T_{Z_{22}} \end{bmatrix} \begin{bmatrix} D_{A_k} & 0\\ 0 & \widehat D_A \end{bmatrix} \begin{bmatrix} T_{W_k}^T & 0 \\[0.5mm] T_{W_{12}}^T & T_{W_{22}}^T \end{bmatrix}\begin{bmatrix} Q_{W_k}^T \\ \wh Q_{W}^T \end{bmatrix}\\
&= Q_{Z_k} T_{Z_k} D_{A_k} T_{W_k}^T Q_{W_k}^T + Q_{Z_k} T_{Z_{12}} \widehat D_A \, \wh T_W^TQ_W^T.
\end{align*}
We observe that 
\begin{align*}
 (I-Q_{Z_k}Q_{Z_k}^T) A &= (A-A_k)- Q_{Z_k} T_{Z_{12}} \widehat D_A \, \wh T_W^TQ_W^T\\
 &= Q_Z{\wh T}_Z \, {\wh D}_A \, {\wh T}_W^T Q_W^T-Q_{Z_k} T_{Z_{12}} \widehat D_A \, \wh T_W^TQ_W^T=\wh Q_Z \ T_{Z_{22}}\widehat D_A \, \wh T_W^TQ_W^T.
\end{align*}
 Consequently, 
\begin{align*}\norm{(I-Q_{Z_k}Q_{Z_k}^T) A}&=\norm{\wh Q_Z T_{Z_{22}}\widehat D_A \, \wh T_W^TQ_W^T}\leq \norm{\wh D_A} \norm{T_{Z_{22}}}\,\norm{\wh T_W} \\
&\leq \alpha_{k+1}\cdot\norm{T_{Z_{22}}}\,\norm{\wh T_W},\end{align*}
and
\[\norm{(I-\mathbb{S})A} \le\alpha_{k+1}\cdot\norm{(S^TQ_{Z_k})^{-1}}\,\norm{T_{Z_{22}}}\,\norm{\wh T_W}.\]
Suppose that $C_A$ and $R_A$ are of full rank. Given the orthogonal projectors $C_AC_A^+$ and $R^+_AR_A$ and computing $M_A$ as $(C_A^T C_A)^{-1}C_A^T A R_A^T(R_A R_A^T )^{-1}=C_A^+AR_A^+$, we have (see \cite[(6)]{Mahoney})
\[A-C_AM_AR_A=A-C_AC_A^+AR_A^+R_A=(I-C_AC_A^+)A+C_AC_A^+A(I-R_A^+R_A).\]
Using the triangle inequality, it follows that \cite{Mahoney,Sorensen}
\[\norm{A-C_AM_AR_A}=\norm{A-C_AC_A^+AR_A^+R_A}\leq \norm{(I-C_AC_A^+)A}+\norm{C_AC_A^+}~\norm{A(I-R_A^+R_A)}.\]
Leveraging the fact that $C_AC_A^+$ is an orthogonal projector so $\norm{C_AC_A^+}=1$, and \cite[Lemma~4.2]{Sorensen}

\[\norm{(I-C_AC_A^+)A}\leq \norm{A(I-\mathbb{P})}~, \quad
\norm{A(I-R_A^+\!R_A)}\leq \norm{(I-\mathbb{S})A},
\]
as a variant of \cite[Thm.~4.8]{gidisu2021} we have
\begin{align*}
 \norm{A-C_AM_AR_A} &\le \alpha_{k+1} \cdot (\norm{{\wh T}_Z}\, \norm{T_{W_{22}}}\,\norm{(Q_{W_k}^TP)^{-1}}+\norm{(S^TQ_{Z_k})^{-1}}\,\norm{T_{Z_{22}}}\,\norm{\wh T_W})\\
 &\le \alpha_{k+1}\cdot(\norm{(Q_{W_k}^TP)^{-1}}+\norm{(S^TQ_{Z_k})^{-1}})\cdot\norm{{\wh T}_Z}\,\norm{\wh T_W}.
\end{align*}
The last inequality follows directly from the fact that the norms of the submatrices of $\wh T_{Z_{22}}$ and $\wh T_{W_{22}}$ are at most $\norm{{\wh T}_Z}$ and $\norm{\wh T_W}$, respectively.
\end{proof}

\Cref{theorem1} suggests that to keep the approximation errors as small as possible, a good index selection procedure that provides smaller quantities $\norm{(U_k^TP_B)^{-1}}$, $\norm{(S_G^TV_k)^{-1}}$, $\norm{(Q_{W_k}^TP)^{-1}}$, and $\norm{(S^TQ_{Z_k})^{-1}}$ would be ideal. The DEIM procedure may be seen as an attempt to attain exactly that. Meanwhile, the quantity $\alpha_{k+1}\cdot\norm{{\wh T}_Z}\,\norm{\wh T_W}$ is a result of the error of the rank-$k$ RSVD. \pg{Additionally, we would like to point out that due to the connection of QDEIM with strong rank-revealing QR 
factorization (matrix volume maximization), the upper bounds for the error constants can be reduced to the theoretically best available one if the QDEIM procedure for the selection of the indices is used instead of the DEIM algorithm \cite{Drmac,drmac2018discrete}.}

Comparing the results of the decomposition of $A$ in \cref{theorem1} to \cite[Thm.~4.1]{Sorensen}, we have that the $\sigma_{k+1}$ in \cite[Thm.~4.1]{Sorensen} is replaced by the error in the RSVD through $\norm{(I-Q_{Z_k}Q_{Z_k}^T)A}$ and $\norm{A(I - Q_{W_k}\, Q_{W_k}^T)}$. Here, $\norm{(Q_{W_k}^TP)^{-1}}$, and $\norm{(S^TQ_{Z_k})^{-1}}$ are computed using the orthonormal bases of the nonsingular matrices from the RSVD of $A$ rather than the singular vectors. Compared with the results in \cite[Thm.~4.8]{gidisu2021} where all quantities are a result of the GSVD, in \cref{theorem1} we have an additional $\norm{\wh T_Z}$, and all quantities are a result of the RSVD.

\section{Numerical Experiments}\label{sec: exper}
In this section, we first evaluate the performance of the proposed RSVD-CUR decomposition for reconstructing a data matrix perturbed with nonwhite noise. We then show how the proposed algorithm can be used for feature selection in multi-view classification problems. In \cref{exp:3}, we only care about the key columns of $B$ and $G$ so we do not explicitly compute the RSVD-CUR factorization. We use the DEIM and the QDEIM index selection scheme interchangeably.

\pg{In \cref{exp:1,exp:2}, we consider two popular covariance structures \cite{wicklin2013simulating} for the correlated noise. The first covariance matrix has a compound symmetry structure ({\rm CSS}), which means the covariance matrix has constant diagonal and constant off-diagonal entries. A homogeneous compound symmetry covariance matrix is of the form 
\begin{equation}\label{cov:CS}
 {\rm CSS}_{\rm cov} =\nu^2\,\smtxa{rrrr}{
 1&\xi&\xi&\xi\\
 \xi&1&\xi&\xi\\
 \xi&\xi&1&\xi\\
 \xi&\xi&\xi&1},
\end{equation}
where $\nu$ and $\xi$ (with $-1<\xi<1$) denote the variance and correlation coefficient, respectively. The second covariance matrix we use has a first-order autoregressive structure ({\rm AR}(1)), which implies the matrix has a constant diagonal and the off-diagonal entries decaying exponentially; it assumes that the correlation between any two elements gets smaller the further apart they are separated (e.g., in terms of time or space). A homogeneous {\rm AR}(1) covariance matrix is of the form 
 \begin{equation}\label{cov:AR}
 {\rm AR}(1)_{\rm cov} =\nu^2\,\smtxa{llll}{
 1&\xi&\xi^2&\xi^3\\
 \xi&1&\xi&\xi^2\\
 \xi^2&\xi&1&\xi\\
 \xi^3&\xi^2&\xi&1}.
\end{equation}}

\begin{experiment}\label{exp:1} {\rm
In our first experiment, we address the problem of matrix perturbation, specifically $A_E = A + BFG$, where $F$ is a Gaussian random matrix and $B$ and $G$ are the Cholesky factors of non-diagonal noise covariance matrices.  The goal is to reconstruct a low-rank matrix $A$ from $A_E$ assuming that the noise covariance matrices or their estimates are known. The requirement that the noise covariance matrices or
their estimates should be known is not always trivial. We evaluate and compare a rank-$k$ RSVD-CUR and CUR decomposition of $A_E$ in terms of reconstructing matrix $A$. The approximation quality of each decomposition is assessed by the relative matrix approximation error, i.e., $\norm{A-\widetilde {A}}\,/\,\norm{A}$, where $\widetilde {A}$ is the reconstructed low-rank matrix. As an adaptation of the first experiment in \cite[Ex.~6.1]{Sorensen} we generate a rank-100 sparse nonnegative matrix $A\in \R^{10000\times 1000}$ of the form
\[A=\sum_{j=1}^{10}\frac{2}{j}\, \bx_j\ \by_j^T + \sum_{j=11}^{100}\frac{1}{j}\, \bx_j\ \by_j^T,\]
where $\bx_j\in \R^{10000}$ and $\by_j\in \R^{1000}$ are random sparse vectors with nonnegative entries (in Matlab, $\bx_j$={\sf sprand}(10000, 1, 0.025) and $\by_j$={\sf sprand}(1000, 1, 0.025)). We then perturb $A$ with a nonwhite noise matrix $BFG$ (see, e.g., \cite[p.~55]{hansen}). The matrix $F\in \R^{10000\times 1000}$ is random Gaussian noise. We assume that $B\in \R^{10000 \times 10000}$ is the Cholesky factor of a symmetric positive definite covariance matrix with compound symmetry structure \eqref{cov:CS} (with diagonal entries 4 and off-diagonal entries 1), and $G\in \R^{1000 \times 1000}$ is the Cholesky factor of a symmetric positive definite covariance matrix with first-order autoregressive structure \eqref{cov:AR} (with diagonal entries 1 and the off-diagonal entries related by a multiplicative factor $\xi= 0.99$). The resulting perturbed matrix we use is of the form $A_E=A+\varepsilon\frac{\norm{A}}{\norm{BFG}}BFG$, where $\varepsilon$ is the noise level. 

Given a noise level, we compare the naive approaches (rank-$k$ SVD-based methods), which do not take the structure of the noise into account with the methods (rank-$k$ RSVD-based methods), which consider the actual noise. That is, we compute the SVD of $A_E$ and the RSVD of $(A_E, B, G)$ to obtain the input matrices for a CUR and an RSVD-CUR decomposition, respectively. For each noise level, we generate ten random cases and take the average of the relative errors for varying $k$ values.

\Cref{fig: 1} summarizes the results of three noise levels (0.1, 0.15, 0.2). We observe that to approximate $A$, the RSVD-CUR factorization enjoys a considerably lower average approximation error than the CUR decomposition. The error of the former is at least half of the latter.
Meanwhile, the average relative error of an RSVD-CUR approximation unlike that of the RSVD (monotonically decreasing) approaches $\varepsilon$ after a certain value of $k$. This situation is natural because the RSVD-CUR routine picks actual columns and rows of $A_E$, so the relative error is likely to be saturated by the noise. The rank-$k$ SVD of $A_E$ fails in approximating $A$ for the given values of $k$. Its average relative error rapidly approaches $\varepsilon$; this is expected since the covariance of the noise is not a multiple of the identity. \pg{The truncated SVD of $A_E$ gives the optimal rank-$k$ approximation to $A_E$. However, here, the goal is to reconstruct the noise-free matrix $A$ hence, we measure the rank-$k$ approximation error against the unperturbed $A$. To heuristically explain why the SVD fails here, let us consider the ideal situation of a perturbation matrix $E$ whose entries are normally distributed with zero mean and standard deviation $\nu$ (white noise). We have that the expected value of $\norm{E}$ is approximately $\nu\sqrt{m}$ \cite{hansen}. Hence if the singular values of unperturbed $A$ decay gradually to zero, then the singular values of the perturbed matrix $A_E$ are expected to decrease monotonically (by definition) until they tend to settle at a level $\tau=\nu\sqrt{m}$ determined by the errors in $A_E$ \cite{hansen}. 
Note that given that the rank of $A$ is $k$, the singular values $(\psi_i)$ of $A_E$ will plateau at $i=k+1$. We can therefore expect to estimate $A$ in terms of the rank-$k$ SVD $A_E$ \cite{hansen}. In this experiment, given the structure of the noise $E=BFG$, the expected value of $\norm{E}$ is not approximately $\nu\sqrt{m}$. As a result, the rank-$k$ SVD of $A_E$ fails to estimate $A$. We need algorithms that can take the actual noise into account, which is typically done by the ``prewhitening" process.  Thus, given the nonsingular error equilibration matrices $B$ and $G$, we have that the error in $B^{-1}A_E\,G^{-1}$ is equilibrated, i.e., uncorrelated with zero mean and equal variance (white noise). One may then approximate $A$ by applying the truncated SVD to the matrix $B^{-1}A_E\,G^{-1}$, followed by a ``dewhitening" by a means of left and right multiplication with $B$ and $G$, respectively \cite{hansen}. However, it is worth noting that even when $B$ and $G$ are nonsingular it may be computationally risky to work with their inverses since they may be close to singular, and so in general it makes sense to reformulate the problem with no inverses. Using the RSVD formulation, this ``prewhitening" and ``dewhitening" process becomes an integral part of the algorithm.

Another observation worth pointing out is that a rank-$k$ CUR of $A_E$ yields a more accurate approximation of $A$ compared to the rank-$k$ SVD of $A_E$, although the column and row indices are selected using the singular vectors of $A_E$. This behavior can be attributed to the fact that the SVD is a linear combination of all the perturbed data points, so the total noise is included in the rank-$k$ approximation. On the other hand, with the CUR approximation, the $C$ and $R$ factors are actual columns and rows of $A_E$, so selecting $k$ columns and rows of $A_E$ implies that our approximation would only contain part and not the total noise. 

We also observe that the approximation accuracy of the QDEIM index selection method is comparable to that of the DEIM scheme, although the former is simpler and more efficient.

It is also worth noting that the improved performance of an RSVD-CUR approximation compared to a CUR factorization is particularly more attractive for higher noise levels with modest $k$ values, i.e., when $k$ is significantly less than rank$(A)$. One can observe from \cref{fig: 1} that the range of $k$ values that yield the lowest approximation error measured against unperturbed $A$ is between 10 and 20. Using the following $k$ values (10, 15, 20), in \cref{table:exp1}, we investigate the significant improvement of a DEIM-RSVD-CUR approximation over a DEIM-CUR factorization. With noise level $0.1$, we see that for the lowest error of a DEIM-CUR, which corresponds to $k=15$, a DEIM-RSVD-CUR produces about a 39\% reduction in the error. The improvement is even more significant if the noise level is $0.2$, where the rank-10 RSVD-CUR approximation reduces the rank-10 CUR decomposition error (lowest error of the CUR) by about 50\%.}
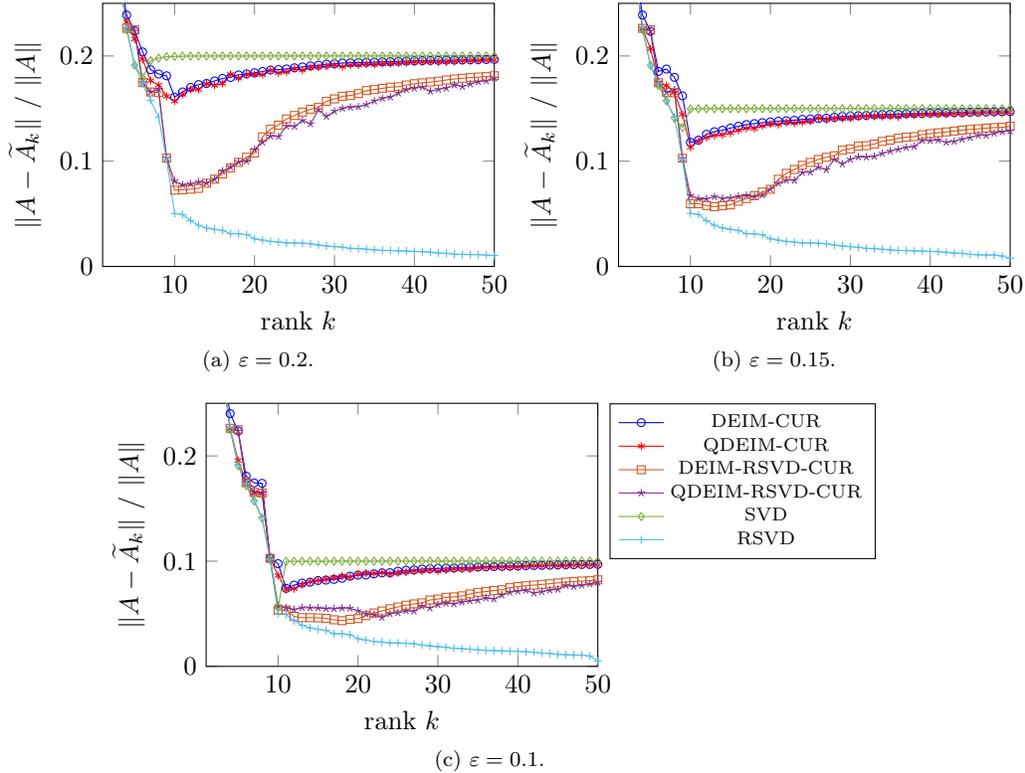
\begin{figure}[htb!]
 \centering
 \subfloat[$\varepsilon=0.2$.]{\label{fig: a}{\input{Plots/RSVD_CUR_plots/ex1_0.2}}}
	\subfloat[$\varepsilon=0.15$.]{\label{fig: b}{\input{Plots/RSVD_CUR_plots/ex1_0.15}}}\\
	\subfloat[$\varepsilon=0.1$.]{\label{fig: c}{\input{Plots/RSVD_CUR_plots/ex1_0.1}}}
	\captionsetup{format=hang}
	\caption{The approximation quality of RSVD-CUR approximations compared with CUR approximations in recovering a sparse, nonnegative matrix $A$ perturbed with nonwhite noise. The average relative errors $\norm{A-\widetilde {A}_k}\,/\,\norm{A}$ (on the vertical axis) as a function of rank $k$ (on the horizontal axis) for $\varepsilon=0.2$, $0.15$, $0.1$, respectively.\label{fig: 1}}
\end{figure}

\begin{table}[htb!]
\centering
\caption{The approximation quality of the DEIM-RSVD-CUR approximations compared with DEIM-CUR approximations in recovering a sparse, nonnegative matrix $A$ perturbed with nonwhite noise. The average relative errors $\norm{A-\widetilde {A}_k}\,/\,\norm{A}$ for $k$ values (10, 15, 20).}\label{table:exp1}\vspace{-2mm}
{\scriptsize \begin{tabular}{clccc} \\ \hline \rule{0pt}{2.3ex}%
{\textbf{Noise}}&{\textbf{Method $\backslash$ \ $k$ }} & 10 & 15 & 20 \\ \hline \rule{0pt}{2.8ex}%
\multirow{2}{*}{$\varepsilon=0.1$}&CUR-$\widetilde {A}_k$ & 0.100 & 0.084 & 0.089 \\
&RSVD-CUR-$\widetilde {A}_k$ & 0.064 & 0.051 & 0.049 \\
\\[-2mm] \hdashline \rule{0pt}{3ex}%
\multirow{2}{*}{$\varepsilon=0.2$}&CUR-$\widetilde {A}_k$ & 0.162 & 0.177 & 0.184 \\
&RSVD-CUR -$\widetilde {A}_k$ & 0.080 & 0.084 & 0.106 \\ \hline
\end{tabular}}
\end{table}

In \cref{fig:3}, using the  DEIM-RSVD-CUR decomposition of $A_E$ with noise level $\varepsilon=0.1$, we show the various quantities in \cref{theorem1}. We observe that the upper bound in \cref{theorem1} may be rather pessimistic, and the true DEIM-RSVD-CUR error may be substantially lower in practice. As in \cite[Fig.~4]{Sorensen}, the magnitude of the quantities $\eta_s$ and $\eta_p$ may vary. We see that $\|\wh T_W\|$ and $\|\wh T_Z\|$ seem to stabilize as $k$ increases and both quantities are close to $\norm{A_E}$.

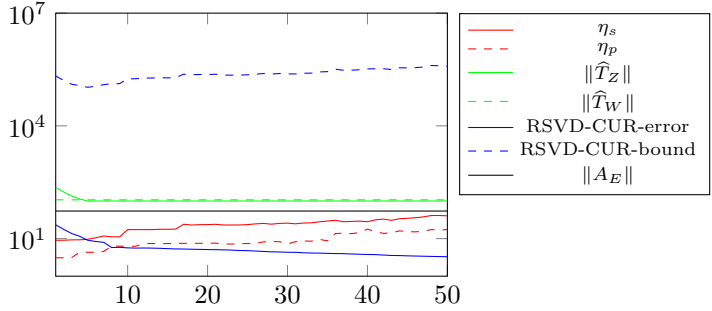
\begin{figure}[ht!]
 \centering
 \input{Plots/RSVD_CUR_plots/boundn}
 \captionsetup{format=hang}
	\caption{Various quantities from \cref{theorem1} using the DEIM index selection scheme: error constants $\eta_p = \norm{(Q_{W}^TP)^{-1}}$ (red dashed) and $\eta_s = \norm{(S_A^TQ_Z)^{-1}}$ (red solid); multiplicative factors $\norm{\wh T_Z}$ (green solid) and $\norm{\wh T_W}$ (green dashed); an RSVD-CUR true error $\|A_E-(CMR)_{\sf rsvd-cur}\|$ of approximating $A_E$ in \cref{exp:1} (blue solid) and its upper bound (blue dashed).\label{fig:3}} 
\end{figure}

}
\end{experiment}
\begin{experiment}\label{exp:2} {\rm
In this experiment, we investigate the use of inexact Cholesky factors $\wh B$ and $\wh G$. In particular, we assume that the exact noise covariance matrices $B^T\!B$ and $G^T\!G$ are unknown, and we compare the performance of the proposed RSVD-CUR decomposition with that of a CUR factorization in reconstructing $A$ from $A_E$. To generate the inexact Cholesky factors $\wh B$ and $\wh G$, we multiply the off-diagonal elements of the exact Cholesky factors $B$ and $G$ by uniform random numbers from the interval $[0.9, 1.1]$. We maintain the experimental setup described in \cref{exp:1} using noise levels $\varepsilon=0.1, 0.2$, except that here we compute the RSVD of $(A_E, \wh B, \wh G)$ to obtain the input matrices for the RSVD-CUR decomposition. Our results, presented in \Cref{fig: 4a,fig: 4b}, suggest that the RSVD and RSVD-CUR factorization still provide good approximation results even when using inexact Cholesky factors. This finding may indicate that we may not necessarily need the exact noise covariance.

\begin{figure}[htb!]
 \centering
 \subfloat[$\varepsilon=0.1$.]{\label{fig: 4b}{\input{Plots/RSVD_CUR_plots/inex_0.1}}}
 \subfloat[$\varepsilon=0.2$.]{\label{fig: 4a}{\input{Plots/RSVD_CUR_plots/inex_0.2}}}
	
	\captionsetup{format=hang}
	\caption{The approximation quality of RSVD-CUR factorizations using inexact Cholesky factors of the noise covariance matrices compared with CUR decompositions in recovering a sparse, nonnegative matrix $A$ perturbed with nonwhite noise. The average relative errors $\norm{A-\widetilde {A}_k}\,/\,\norm{A}$ (on the vertical axis) as a function of rank $k$ (on the horizontal axis) for $\varepsilon=0.1$, $0.2$, respectively.\label{fig: 4}}
\end{figure}
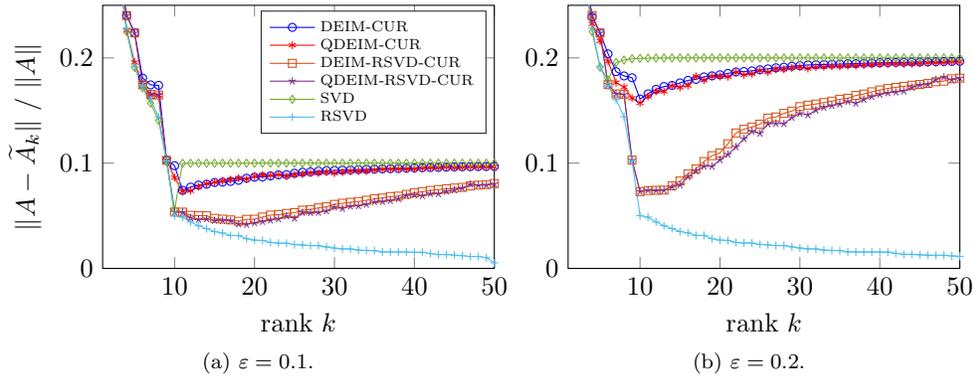
}
\end{experiment}

\begin{experiment}\label{exp:3} {\rm This experiment demonstrates the effectiveness of an RSVD-ID (as defined in \eqref{RSVD:ID}) in discovering the underlying class structure that two views of a data set share. We demonstrate that using an RSVD-ID as a feature selection technique in multi-view classification problems can improve classification accuracy. Typically, multi-view classification problems aim to improve classification accuracy by integrating information from different views into a unified representation. Two traditional approaches for dimension reduction in such problems are concatenation and separation. The concatenation strategy merges different views into a new feature space and applies traditional feature selection algorithms such as an interpolative decomposition on the merged set, while the separation strategy performs feature selection separately on each view. The concatenation approach may overlook the unique statistical properties of each feature set, while the separation strategy may miss important relationships between views. Given that multiple views of data can offer complementary information, it is reasonable to develop a feature selection algorithm that leverages all views and exploits their relationships.

Consider the two views/feature sets as matrices $B$ and $G$. We are primarily interested in the key columns of $B$ and $G$ rather than their explicit CUR factorization. An RSVD-ID may serve as an unsupervised feature selection method for two-view data sets, leveraging the correlation between the views. While we are interested in a subset of the columns of $B$ and $G$, the problem involves three matrices and requires the use of RSVD-ID. Specifically, we use $A:= B^TG$ as the cross-correlation between the two views, $B$ as View1, and $G$ as View2. We compare the classification test error rate of the QDEIM-type RSVD-ID scheme with that of the QDEIM-type ID algorithm. (Recall the relationship between {\sf cca}$(B, G)$ and the RSVD of $(B^TG, B^T, G)$ as discussed in \cref{sec: RSVD}.)  We generate several reduced feature sets and compare their classification performance as follows:
\begin{enumerate}[label=(\roman*)]
 \item Two reduced feature sets are created by applying the QDEIM-ID procedure on each view separately, i.e., the separation strategy. We label the reduced feature sets as ID-View1 and ID-View2.
 
 \item Another two reduced feature sets are created by performing QDEIM-type RSVD-ID on the two views, i.e., RSVD-ID of $(A, B^T, G)$. We label them RSVD-ID-View1 and RSVD-ID-View2, which are the RSVD-ID selected features of views 1 and 2, respectively.
 
 \item The feature set labeled Fused-ID is a concatenation of the two reduced feature sets from (i), i.e., [ID-View1, ID-View2].
 
 \item Concat-ID is another feature set formed by applying the QDEIM-ID scheme on the column concatenation of both views. Note that here, the concatenation of the views is done before the dimension reduction is performed, i.e., the concatenation strategy.
 
 \item Finally, we concatenate the two reduced feature sets from (ii) to get Fused-RSVD-ID, i.e., [RSVD-ID-View1, RSVD-ID-View2].
\end{enumerate}

To ensure a fair comparison, we present the results of the single views (i) and (ii) in one table and the feature fusion results (iii), (iv), and (v) in a separate table. This allows us to investigate the impact of incorporating complementary information from all views on the classification performance of each feature set and to determine which method yields the best results.

To demonstrate the effectiveness of our approach, we use the handwritten digits data set from the UCI repository, which contains features of handwritten numerals (0--9) extracted from Dutch utility maps. The data set consists of 2000 digits, with 200 instances for each of the ten classes. We extract three types of feature sets: Fourier descriptors, Karhunen--Loève coefficients, and image vectors. The Fourier set contains 76 two-dimensional shape descriptors, the Karhunen--Loève feature set consists of 64 features, and the `pixel' feature set was obtained by dividing the image of $30\times 48$ pixels into 240 tiles of $2\times 3$ pixels. We combine these three feature sets to form three experiments of a two-view classification. In the first experiment, we use the Fourier coefficients of the character shapes ({\sf fou}) and the Karhunen--Loève coefficients ({\sf kar}) as view-1 and view-2, respectively. The second experiment uses the pixel averages in $2\times 3$ windows ({\sf pix}) as the first view and the Fourier coefficients of the character shapes ({\sf fou}) as the second view. Finally, the third experiment takes the pixel averages in $2\times 3$ windows ({\sf pix}) as the first view and the Karhunen--Loève coefficients ({\sf kar}) as the second view. \cref{MVDs} summarizes the basic traits of the various data sets. We normalize all the data sets to have a zero center and a standard deviation of one. 

\begin{table}[htb!]
\centering
\caption{Summary characteristics of multi-view data sets used in the experiments.}\label{MVDs}\vspace{-4mm}
{\scriptsize \begin{tabular}{lccc} \\ \hline \rule{0pt}{2.3ex}%
{\bf Data set} & {\bf Samples} & {\bf View 1} $(B)$ & {\bf View 2} $(G)$ \\ \hline \rule{0pt}{2.3ex}%
Digits ({\sf fou vs.~kar}) & 2000 & \phantom{1}76 & \phantom{1}64 \\
Digits ({\sf pix vs.~fou}) & 2000 & 240 & \phantom{1}76 \\
Digits ({\sf pix vs.~kar}) & 2000 & 240 & \phantom{1}64 \\ \hline 
\end{tabular}}
\end{table}

In each experiment, we randomly split the normalized data into 75\% training and 25\% testing data. For the randomization of the experiments, we perform 20 cases using different random seeds. \cref{fig:5,fig:6} reports the average classification test error rate of the default $k$-nearest neighbor ($k$-NN) classifier in MATLAB for varying reduced dimensions.
\begin{table}[htb!]
\centering
\caption{The average classification test error rate over 20 different random train-test splits of the `pixel', Fourier descriptors and Karhunen--Loève feature sets using QDEIM-ID and QDEIM-RSVD-ID as a dimension reduction method for a k-nearest neighbor classifier in \cref{exp:3}.}\label{fig:5}\vspace{-4mm}
{\scriptsize \begin{tabular}{lccccc} \\ \hline \rule{0pt}{2.3ex}%
Data/Method & Rank-$k$ & ID-View 1 & RSVD-ID-View1 & ID-View2 & RSVD-ID-View2 \\ \hline \rule{0pt}{2.3ex}%
\multirow{2}{*}{$B$={\sf pix} vs. $G$={\sf fou}} & 20 & 0.15 & \textbf{0.10}  & 0.33  & \textbf{0.19} \\
  & 30 & 0.10 & \textbf{0.07}  & 0.28  & \textbf{0.19} \\
  \\[-1.5mm]
\multirow{2}{*}{$B$={\sf fou} vs. $G$={\sf kar}} & 20 & 0.33 & \textbf{0.18}  & 0.17  & \textbf{0.07} \\
 & 30 & 0.28 & \textbf{0.19}  & 0.13  & \textbf{0.06} \\
 \\[-1.5mm]
$B$={\sf pix} vs. $G$={\sf kar} & 20 & 0.15 & \textbf{0.08}  & 0.17  & \textbf{0.04}  \\ & 30 & 0.10 & \textbf{0.06}  & 0.14  & \textbf{0.04} \\ \hline 
\end{tabular}}
\end{table}

\begin{table}[htb!]
\centering
\caption{The average classification test error rate over 20 different random train-test splits of fused feature sets using QDEIM-ID and QDEIM-RSVD-ID as a dimension reduction method for a k-nearest neighbor classifier in \cref{exp:3}.}\label{fig:6}\vspace{-4mm}
{\scriptsize \begin{tabular}{lccccc} \\ \hline \rule{0pt}{2.3ex}%
Data/Method & Rank-$k$ & Fused-ID & Concat-ID & Fused-RSVD-ID \\ \hline \rule{0pt}{2.3ex}%

\multirow{2}{*}{$B$={\sf pix} vs. $G$={\sf fou}} & 20 & 0.12  & 0.13 & \textbf{0.06}  \\
& 30 & 0.09 &  0.11 & \textbf{0.04} \\
\\[-1.5mm]
\multirow{2}{*}{$B$={\sf fou} vs. $G$={\sf kar}}  & 20 & 0.14  & 0.15 & \textbf{0.03} \\
& 30 & 0.10  & 0.13 & \textbf{0.02} \\
\\[-1.5mm]
$B$={\sf pix} vs. $G$={\sf kar} & 20 & 0.10  & \textbf{0.05} & 0.06  \\
& 30 & 0.07  & \textbf{0.04} & \textbf{0.04} \\ \hline 
\end{tabular}}
\end{table}
}
 From the results, we observe that a QDEIM-RSVD-ID method consistently performs better than a QDEIM-ID scheme. In particular, from the classification results using single views, the QDEIM-RSVD-ID significantly improves the worse QDEIM-ID single view results, as seen in columns 3 and 4 of \cref{fig:5}.
We notice that using information from multiple views indeed reduces the classification test error rate. Furthermore, in \cref{fig:6}, we observe that feature fusion from a QDEIM-type RSVD-ID approximation usually gives the least test error rate compared with the two other approaches involving the QDEIM-ID scheme, i.e., method (iii) and (iv). 

\end{experiment}

\begin{experiment}\label{exp:5}{\rm In certain applications, selecting the ``best" feature subset is not enough; cost considerations associated with those features also need to be taken into account. For instance, in medical diagnosis, medical tests incur a cost and risk, whereas symptoms observed by patients or medical practitioners are usually cost-free. Thus, reducing monetary costs and sparing patients from unpleasant or dangerous clinical tests (which can be quantified as costly variables) is essential. In image analysis, feature acquisition processes' time and space complexities generally constitute the computational cost of features. As such, reducing this cost by selecting only relevant and ideally ``inexpensive" variables is a typical modeler's goal.

In this experiment, we evaluate the efficacy of the CUR, the GCUR, and the RSVD-CUR methods in selecting relevant features that enhance prediction accuracy while keeping feature acquisition costs low in the presence of correlated noise. We demonstrate how RSVD-CUR factorization can be used in this context using the Thyroid disease data set from the UCI repository. The problem is to determine whether a patient referred to the clinic has hypothyroidism. The data set comprises three classes: normal (not hypothyroid), hyperfunction, and subnormal functioning. Given that 92\% of the patients are not hyperthyroid, a good classifier must have an accuracy significantly higher than 92\%. 

This 21-dimensional data set has a separate training and testing set. The training set consists of 3772 samples and the testing set consists of 3428 instances. The data set comes with an intrinsic cost associated with 20 input features, which we used to construct a diagonal matrix $G\in\R^{20\times20}$. We, therefore, dropped the feature that does not have an associated cost. \cref{thyroidattr} show the Thyroid data set's attributes. For our purpose, we add a small correlated noise perturbation to the normalized training data set, e.g., $\varepsilon=\norm{A_E-A}\,/\,\norm{A}=0.001$, where matrix $A$ represents the original 20-dimensional training data set. (Note that we only perturb the training data set.) We assume that the lower triangular matrix $B\in \R^{3772 \times 3772}$ is the Cholesky factor of a symmetric positive definite matrix with first-order autoregressive structure \eqref{cov:AR} (with diagonal entries 1 and the off-diagonal entries related by a multiplicative factor of 0.99). So the perturbation matrix $E=B\widetilde E$, where $\widetilde E\in\R^{3772\times 20}$ is a random Gaussian matrix. 
\begin{table}[htb!]
\centering
\caption{Thyroid disease data features and their associated costs.}\label{thyroidattr}\vspace{-4mm}
{\scriptsize \begin{tabular}{lcllr} \\ \hline \rule{0pt}{2.3ex}%
Feature & Cost &  & Feature & Cost  \\ \hline \rule{0pt}{2.3ex}%
age & 1.00 &  & query\_hyperthyroid & 1.00  \\
sex & 1.00 &  & lithium & 1.00  \\
on\_thyroxine  & 1.00 &  & goitre & 1.00  \\
query\_on\_thyroxine & 1.00 &  & tumor  & 1.00  \\
on\_antithyroid\_medication & 1.00 &  & hypopituitary & 1.00  \\
sick  & 1.00 &  & psych  & 1.00  \\
pregnant  & 1.00 &  & TSH & 22.78 \\
thyroid\_surgery & 1.00 &  & T3 & 11.41 \\
I131\_treatment & 1.00 &  & TT4 & 14.51 \\
query\_hypothyroid & 1.00 &  & T4U & 11.41  \\ \hline 
\end{tabular}}
\end{table}

 We evaluate the performance of the three algorithms based on the cost of features selected and their classification test error rates. The DEIM index selection procedure is used in this experiment. The standard CUR decomposition selects ten columns of $A_E$ without considering the noise filter $B$ and the cost matrix $G$. The GCUR method selects a subset of ten columns of $A_E$ relative to $G$ but does not consider the noise. The RSVD-CUR method selects ten columns of $A_E$ by incorporating all available prior information from $(A_E, B, G)$. To ensure the validity of the results, we perform ten cases using different random seeds. \cref{thyroidresults} presents the average total cost of the selected features and the average classification test error rate, where we use the default $k$-nearest neighbor ($k$-NN) classifier in MATLAB.
\begin{table}[htb!]
\centering
\caption{Average classification test error rate and total cost of selected variables using data set in \cref{exp:5}. The $k$-NN model is trained using the perturbed training data set.}\label{thyroidresults}\vspace{-4mm}
{\scriptsize \begin{tabular}{lccll} \\ \hline \rule{0pt}{2.3ex}%
Method/Criteria & \multicolumn{2}{c}{Error rate} &  & Cost  \\& $\varepsilon=10^{-4}$ & $\varepsilon=10^{-3}$ & &\\\hline \rule{0pt}{2.3ex}%
RSVD-CUR & 0.07 & 0.07 & & 10  \\
GCUR & 0.09 & 0.26 & & 10  \\
CUR & 0.11 & 0.11 & & 23.51 \\
All features & 0.07 & 0.07 & & 76.11  \\ \hline 
\end{tabular}}
\end{table}
The results shown in \cref{thyroidresults} indicate that methods incorporating cost information (i.e., GCUR and RSVD-CUR) lead to lower total costs of features. The features selected by the RSVD-CUR result in the lowest classification error rate, possibly because the RSVD-CUR considers the noise structure during feature selection. When $\varepsilon=10^{-4}$, the average classification error rate of the RSVD-CUR is approximately 28\% and 57\% lower than that of the GCUR and the CUR, respectively. Furthermore, the error rates of the RSVD-CUR and the CUR are less sensitive to perturbation levels compared to the GCUR (the error rate of the GCUR-selected features increases drastically from 0.09 to 0.26 as the noise level increases from 0.0001 to 0.001). Notably, the error rate of using the full 20-dimensional feature set is similar to that of the RSVD-CUR's 10-dimensional feature set with lower cost. When $\varepsilon=10^{-3}$, both the RSVD-CUR and the GCUR select the ``pregnant" feature (with lower cost) as the most important, while the CUR selects ``TT4" (with higher cost) as the most important, resulting in higher feature cost compared to the GCUR and the RSVD-CUR.
}
\end{experiment}

\textbf{General Gauss-Markov model with constraints.}
We briefly describe another possible application of the RSVD-CUR factorization here. 

The RSVD-CUR decomposition may be used as a subset selection procedure in the general Gauss-Markov linear models with constraints problem \eqref{eq:guass-markov}. ``This problem formulation admits ill-conditioned or rank-deficient $B\in\R^{m\times \ell}$ and $G\in\R^{d\times n}$ (usually with $d\le n$) matrices" \cite{de1991restricted}. The matrix $B$ may be considered a noise filter and $G$ may represent prior information about the unknown components of $\bx$ or may reflect the fact that certain components of $\bx$ are more important or less costly than others \cite{de1991restricted}.  Minimizing $\norm{\by}^2 +\norm{\bff}^2$ reflects that the goal is to explain as much in terms of the columns of $A$ (i.e., minimize $\norm{\by}^2$), taking into consideration the prior information on the structure of the noise as well as the preference of the modeler to use more of one predictor than others in explaining the phenomenon \cite{de1991restricted}. 
It is easy to see that the problem has a solution if the linear system 
\[\mtxa{cc}{A&B\\G&0}\mtxa{c}{\bx\\\by}=\mtxa{c}{\bb\\\bff}\]
is consistent.

In many applications, it is desirable to reduce the number of variables that are to be considered or measured in the future. As a result, it would be appropriate to use a variable subset selection method that incorporates all available prior information (i.e., $B$ and $G$). Since this problem \eqref{eq:guass-markov} involves three matrices, the RSVD-CUR is a suitable procedure for variable subset selection. One may argue that a CUR decomposition of $B^{-1}AG^{-1}$  may be used if $B$ and $G$ are nonsingular. However, since the above problem admits an ill-conditioned or rank-deficient $B$, such a formulation may not always be valid. Suppose we want to select $k$ columns of $A$ and the corresponding columns of $G$, the above linear system reduces to 
\[\mtxa{cc}{A_k&B\\G_k&0}\mtxa{c}{\bx_k\\\by}=\mtxa{c}{\wh \bb\\ \wh \bff},\]
where $A_k=A(:,\bp)$ and $G_k=G(:,\bp)$. The index vector $\bp$ is obtained by applying \cref{algo: RSVD-CUR-DEIM} to $(A, B, G)$.  
\begin{example}\label{ex:5.5}{\rm
The following examples are adapted from \cite{cox1999row}.
We give results for three problems with $m = \ell = 1000$ and $n=d = 100$. We denote by {\sf randn} a matrix or vector from the standard normal distribution and by ${\sf randsvd}(\kappa)$ a random matrix with spectral norm condition number $\kappa$ and geometrically distributed singular values; generated by the routine {\sf randsvd} in Matlab's {\sf gallery}. We consider problems where either one of the matrices is ill conditioned. For all problems, we take $\bb={\sf randn}$ and $\bff={\sf randn}$. \cref{markovproblem} reports the average relative errors of 100 test cases for each problem using the original and reduced system. We compute the errors as $\norm{(\bb, \bff)-(\wh \bb, \wh \bff)}$. In this RSVD-CUR type approach for a Gauss--Markov application, typical behavior of slowly decaying decomposition error is observed.

\begin{table}[htb!]
\centering
\caption{Average relative errors of the original and reduced system for varying $k$  values for the various problems of \cref{ex:5.5}.} \label{markovproblem}\vspace{-4mm}
{\scriptsize \begin{tabular}{lll|ccc}
\hline
\multicolumn{3}{c|}{Problem \  $\backslash$\ $k$} & 10 & 20 & 30 \\
$A$ & $B$ & $G$ \\ \hline \rule{0pt}{2.7ex}%
{\sf randsvd}(10) & {\sf randsvd}($10^4$) & {\sf randn} & 0.29                      & 0.27                  & 0.25 \\
{\sf randsvd}($10^6$) &  {\sf randsvd}(10)  & {\sf randn} & 0.29                      & 0.27                  & 0.25  \\
{\sf randsvd}($10^4$) & {\sf randsvd}($10^4$) & {\sf randsvd}(10) &0.28                      & 0.27                  & 0.25 \\
{\sf randn} & {\sf randn} & {\sf randn} & 0.29                      & 0.27                  & 0.25 \\ \hline
\end{tabular}}
\end{table}

}
\end{example}

\section{Conclusions}\label{sec: con}
We have proposed a new low-rank matrix decomposition, referred to as RSVD-CUR factorization, which extends the CUR decomposition to matrix triplets $(A, B, G)$. The construction of the $C$ and $R$ factors is performed using the DEIM or QDEIM index selection procedure, although other selection methods such as a maximum volume algorithm \cite{Goreinov2010} may be used instead. We have discussed the relationship between this DEIM type RSVD-CUR and a DEIM type CUR or GCUR for nonsingular $B$ and $G$. When $B=G=I$, the RSVD-CUR decomposition of $A$ coincides with a CUR decomposition of $A$. Additionally, when $B=I$, the DEIM type RSVD-CUR of $(A, B, G)$ corresponds to a DEIM type GCUR of $(A, G)$, and similarly for the transpose of $(A, B)$ when $G=I$. The RSVD-CUR factorization can be applied to feature fusion and feature subset selection in multi-view classification and multi-label classification problems. In data perturbation problems, the RSVD-CUR approximation can provide more accurate results than a CUR factorization when reconstructing a low-rank matrix from a correlated noise-perturbed data matrix. The RSVD-CUR factorization can also be used as a subset selection technique in generalized Gauss-Markov problems with constraints. The experiments in \cref{sec: exper} demonstrate the effectiveness of the RSVD-CUR factorization in these applications.

\section*{Acknowledgment}
The authors thank Ian Zwaan for several helpful discussions on algorithms for computing a restricted singular value decomposition. The authors would also like to express special thanks to the referees and editor, whose helpful expert suggestions led to significant improvements in the content and presentation of the paper.

\end{document}

%% file: Plots/RSVD_CUR_plots/ex1_0.2.tex
%
%
\definecolor{mycolor1}{rgb}{0.00000,0.44700,0.74100}%
\definecolor{mycolor2}{rgb}{0.85000,0.32500,0.09800}%
\definecolor{mycolor3}{rgb}{0.92900,0.69400,0.12500}%
\definecolor{mycolor4}{rgb}{0.49400,0.18400,0.55600}%
\definecolor{mycolor5}{rgb}{0.46600,0.67400,0.18800}%
\definecolor{mycolor6}{rgb}{0.30100,0.74500,0.93300}%
\begin{tikzpicture}

\begin{axis}[%
width=0.4\textwidth,
height=3.5cm,
at={(0\textwidth,0\textwidth)},
scale only axis,
xmin=1,
xmax=50,
ymin=0,
ymax=0.25,
xlabel=rank $k$,
ylabel={$\| A - \widetilde{A}_k \|\ /\ \| A \|$},
axis background/.style={fill=white}
]
\addplot [color=blue, mark=o, mark size=1.5pt,mark options={solid, blue}, forget plot]
  table[row sep=crcr]{%
1	0.618589639170852\\
2	0.37383795225893\\
3	0.284852787669925\\
4	0.238909243728072\\
5	0.223833607872855\\
6	0.203850677040055\\
7	0.18695656740094\\
8	0.182880149424912\\
9	0.181158163733663\\
10	0.160905415797211\\
11	0.16581108516302\\
12	0.170159193927934\\
13	0.172806002698074\\
14	0.173826396162659\\
15	0.176247051490059\\
16	0.178662427834773\\
17	0.179622936633376\\
18	0.181981889641854\\
19	0.182956360820265\\
20	0.184029940009212\\
21	0.185246068649812\\
22	0.186681084257055\\
23	0.187074851141511\\
24	0.187657297457112\\
25	0.188750142217667\\
26	0.18958197570928\\
27	0.190593745578657\\
28	0.19118173395958\\
29	0.192146110766377\\
30	0.192624074077385\\
31	0.193293722775025\\
32	0.193553082570238\\
33	0.193773953233855\\
34	0.193968877643802\\
35	0.194143346517406\\
36	0.194451237176449\\
37	0.194631481838725\\
38	0.194946349561951\\
39	0.195139361956872\\
40	0.195399690262316\\
41	0.195560330446947\\
42	0.195699402170255\\
43	0.19585860715172\\
44	0.196067904585662\\
45	0.196101381610148\\
46	0.196219934953358\\
47	0.196382360629493\\
48	0.196508980810189\\
49	0.196629560417201\\
50	0.196768685428789\\
};
\addplot [color=red, mark=asterisk, mark size=1.5pt,mark options={solid, red}, forget plot]
  table[row sep=crcr]{%
1	0.618589639170852\\
2	0.373861181336824\\
3	0.284870754356218\\
4	0.232414767665696\\
5	0.216050876297982\\
6	0.196766170838173\\
7	0.176530924723648\\
8	0.172266971547772\\
9	0.161721688031019\\
10	0.157123873247885\\
11	0.163436172579086\\
12	0.16818960851549\\
13	0.168144985972969\\
14	0.173415411516413\\
15	0.172047233646717\\
16	0.173820412335712\\
17	0.182551206466406\\
18	0.178849476829185\\
19	0.182323400715927\\
20	0.182404434512478\\
21	0.182165449086699\\
22	0.186376589986489\\
23	0.18394048874869\\
24	0.184986123621969\\
25	0.187915088136146\\
26	0.186521334266127\\
27	0.1891737224253\\
28	0.189173196422401\\
29	0.190543131250562\\
30	0.191195937802008\\
31	0.189885379519879\\
32	0.191279175708339\\
33	0.19145002396696\\
34	0.192097916623212\\
35	0.191866048233116\\
36	0.192576355119453\\
37	0.192633571878952\\
38	0.193049538622912\\
39	0.193630055988868\\
40	0.194096518172343\\
41	0.194317293085872\\
42	0.19449389659018\\
43	0.194287874822646\\
44	0.194353521446331\\
45	0.194836622444348\\
46	0.195364264692426\\
47	0.195494000967541\\
48	0.195652488499528\\
49	0.196012696211215\\
50	0.196112629289978\\
};
\addplot [color=mycolor2, mark=square, mark size=1.5pt,mark options={solid, mycolor2}, forget plot]
  table[row sep=crcr]{%
1	0.618589639170852\\
2	0.373861181336824\\
3	0.284912311088637\\
4	0.226244848825371\\
5	0.224867833493003\\
6	0.17467035749543\\
7	0.165810121865829\\
8	0.165237995498659\\
9	0.102941268309881\\
10	0.0724580607357817\\
11	0.0727838105834991\\
12	0.0733995620235099\\
13	0.0741640708530386\\
14	0.0789705381757431\\
15	0.0828829110695387\\
16	0.0878388664390992\\
17	0.093514253812125\\
18	0.0988491751492738\\
19	0.103709841240095\\
20	0.107713229964876\\
21	0.123072347000306\\
22	0.130439510405698\\
23	0.134311240268076\\
24	0.140151508186985\\
25	0.143447661227305\\
26	0.146669361043409\\
27	0.14940204748342\\
28	0.151121921331371\\
29	0.156907530990929\\
30	0.159183589541415\\
31	0.161726301945877\\
32	0.162401830588447\\
33	0.164217166203031\\
34	0.167090670950585\\
35	0.168093352356978\\
36	0.169255676107479\\
37	0.170343659783156\\
38	0.170746882442423\\
39	0.172429213358817\\
40	0.173604278778176\\
41	0.174457777901563\\
42	0.175268274234163\\
43	0.176035285722865\\
44	0.177215049904382\\
45	0.178136419443268\\
46	0.17878741576379\\
47	0.179408938907805\\
48	0.179946710529333\\
49	0.180376460771528\\
50	0.181227453751268\\
};

\addplot [color=mycolor4, mark=star, mark size=1.5pt,mark options={solid, mycolor4}, forget plot]
  table[row sep=crcr]{%
1	0.618589639170852\\
2	0.373861181336824\\
3	0.284912311088637\\
4	0.226244848825371\\
5	0.223717298395481\\
6	0.17467035749543\\
7	0.165810121865829\\
8	0.168294458043855\\
9	0.102943684605347\\
10	0.0812284128553158\\
11	0.0774747382442519\\
12	0.0784209717402609\\
13	0.08028105589695\\
14	0.0792786452805778\\
15	0.0828919539569193\\
16	0.0905588075237506\\
17	0.0945036929426316\\
18	0.0992476034682296\\
19	0.100077361966034\\
20	0.111066853360971\\
21	0.118058617329169\\
22	0.123833416265832\\
23	0.123443370169834\\
24	0.133290220787122\\
25	0.133107691532723\\
26	0.138722873294551\\
27	0.135482572276934\\
28	0.148946458775306\\
29	0.142485665134404\\
30	0.147509144814071\\
31	0.150243050430191\\
32	0.151279522931748\\
33	0.151734348912155\\
34	0.15387164574854\\
35	0.157214613037472\\
36	0.159769789305652\\
37	0.160911306051931\\
38	0.164632616391404\\
39	0.168044062243954\\
40	0.169336794976793\\
41	0.170531111626049\\
42	0.166445243566061\\
43	0.168341273681787\\
44	0.169415721117421\\
45	0.172374046085341\\
46	0.170892100242053\\
47	0.173378340075692\\
48	0.175430561797083\\
49	0.176410595309954\\
50	0.178989834819163\\
};
\addplot [color=mycolor5, mark=diamond, mark size=1.5pt,mark options={solid, mycolor5}, forget plot]
  table[row sep=crcr]{%
1	0.618499966109694\\
2	0.373737886096262\\
3	0.284012738863235\\
4	0.225003687985915\\
5	0.191163878994616\\
6	0.178972037271361\\
7	0.195640730669385\\
8	0.197986767016056\\
9	0.199237817550365\\
10	0.199623249209613\\
11	0.199748592660204\\
12	0.199904905073718\\
13	0.199937598507671\\
14	0.199949049801451\\
15	0.19998784947183\\
16	0.199988733273918\\
17	0.199988768002029\\
18	0.199988927560884\\
19	0.199989259455013\\
20	0.199989792020422\\
21	0.199989683788683\\
22	0.199991588345755\\
23	0.19999310421258\\
24	0.199995421820543\\
25	0.199995781270173\\
26	0.199996205806492\\
27	0.199996418229242\\
28	0.199996382943967\\
29	0.199997002298205\\
30	0.199996703110982\\
31	0.199997448671347\\
32	0.199997298891791\\
33	0.199998213970143\\
34	0.199998320832974\\
35	0.199998308081133\\
36	0.199998412320141\\
37	0.199998551138417\\
38	0.199998814643202\\
39	0.19999881940616\\
40	0.199999127956214\\
41	0.199999151991303\\
42	0.199999224981078\\
43	0.199999360307343\\
44	0.199999305350886\\
45	0.199999469600652\\
46	0.199999529349299\\
47	0.199999546016866\\
48	0.199999557088954\\
49	0.19999960018044\\
50	0.199999621385493\\
};
\addplot [color=mycolor6, mark=+, mark size=1.5pt, mark options={solid, mycolor6}, forget plot]
  table[row sep=crcr]{%
1	0.619089085045264\\
2	0.373758135716215\\
3	0.285591927471099\\
4	0.227018998640471\\
5	0.191328487478833\\
6	0.173772700593925\\
7	0.157707774440683\\
8	0.1419318513861\\
9	0.101577608343697\\
10	0.050355135068876\\
11	0.0493124431249677\\
12	0.0434761470089909\\
13	0.0390686843189686\\
14	0.0365414830638502\\
15	0.0352530631189007\\
16	0.034216294805891\\
17	0.0310825884989954\\
18	0.0310750518734763\\
19	0.0301412173738691\\
20	0.026249338773368\\
21	0.0250383957126139\\
22	0.0237545431099753\\
23	0.0231829960145032\\
24	0.022317381533851\\
25	0.0223030121665802\\
26	0.0221221377554589\\
27	0.0215043869649338\\
28	0.0202806769513426\\
29	0.0194831881496689\\
30	0.0188103104825503\\
31	0.0182015669267907\\
32	0.0171955202011606\\
33	0.0169224385825558\\
34	0.0163274586820729\\
35	0.0157107041306101\\
36	0.0152184339637\\
37	0.0151399791008567\\
38	0.0148157544688605\\
39	0.0147186173388337\\
40	0.0142528931687045\\
41	0.0140246805724919\\
42	0.0135113350263316\\
43	0.0126805490283582\\
44	0.0124216468435125\\
45	0.0116807373665936\\
46	0.01144076206586\\
47	0.0112568244531702\\
48	0.0111826880751008\\
49	0.0107131555309434\\
50	0.010323940290101\\
};
\end{axis}

\end{tikzpicture}%

%% file: Plots/RSVD_CUR_plots/ex1_0.15.tex
%
%
\definecolor{mycolor1}{rgb}{0.00000,0.44700,0.74100}%
\definecolor{mycolor2}{rgb}{0.85000,0.32500,0.09800}%
\definecolor{mycolor3}{rgb}{0.92900,0.69400,0.12500}%
\definecolor{mycolor4}{rgb}{0.49400,0.18400,0.55600}%
\definecolor{mycolor5}{rgb}{0.46600,0.67400,0.18800}%
\definecolor{mycolor6}{rgb}{0.30100,0.74500,0.93300}%
\begin{tikzpicture}

\begin{axis}[%
width=0.4\textwidth,
height=3.5cm,
at={(0\textwidth,0\textwidth)},
scale only axis,
xmin=1,
xmax=50,
ymin=0,
ymax=0.25,
xlabel=rank $k$,
ylabel={$\| A - \widetilde{A}_k \|\ /\ \| A \|$},
axis background/.style={fill=white}
]
\addplot [color=blue, mark=o, mark size=1.5pt,mark options={solid, blue}, forget plot]
  table[row sep=crcr]{%
1	0.618587589953284\\
2	0.373834875390268\\
3	0.284858898643265\\
4	0.238887068033857\\
5	0.223390888856427\\
6	0.185315581734372\\
7	0.187559103916704\\
8	0.179701758923256\\
9	0.162080593979779\\
10	0.117904064076142\\
11	0.119998815851982\\
12	0.125569227620016\\
13	0.128014903627897\\
14	0.129390815204126\\
15	0.131479637349047\\
16	0.133233192196865\\
17	0.134780919377073\\
18	0.135709946446332\\
19	0.136635786631301\\
20	0.137357466006346\\
21	0.138005357085233\\
22	0.138451515700301\\
23	0.138873464432135\\
24	0.139740740758507\\
25	0.14009594407504\\
26	0.140570419414833\\
27	0.140911385591412\\
28	0.141717406343482\\
29	0.142572399703126\\
30	0.142812914368636\\
31	0.143177721633176\\
32	0.143502407436994\\
33	0.144033158768804\\
34	0.14433764520909\\
35	0.144572984115473\\
36	0.144711861726496\\
37	0.145078096640535\\
38	0.145248332807801\\
39	0.145407096458277\\
40	0.145710773132248\\
41	0.145836491906285\\
42	0.146020331035124\\
43	0.146165024441756\\
44	0.146258763987734\\
45	0.146417168714881\\
46	0.146580444932807\\
47	0.146708003532463\\
48	0.14690615022877\\
49	0.147048233882721\\
50	0.147115027191353\\
};
\addplot [color=red, mark=asterisk, mark size=1.5pt,mark options={solid, red}, forget plot]
  table[row sep=crcr]{%
1	0.618587589953284\\
2	0.373853908171904\\
3	0.284871763158995\\
4	0.226186612179506\\
5	0.206669505571951\\
6	0.174649731175065\\
7	0.171884747707265\\
8	0.165072950716207\\
9	0.144401055481038\\
10	0.113206990674608\\
11	0.118991781660117\\
12	0.122583779293344\\
13	0.124418875194998\\
14	0.125087805395372\\
15	0.126753893547302\\
16	0.12941359097705\\
17	0.131254829663152\\
18	0.130989434762338\\
19	0.133472681089913\\
20	0.135236325347127\\
21	0.134836114447926\\
22	0.137109273533508\\
23	0.136143787370237\\
24	0.136950272129649\\
25	0.13782059118429\\
26	0.140891177850368\\
27	0.139583037244779\\
28	0.140659175421078\\
29	0.140683046540682\\
30	0.140848717148989\\
31	0.142145278952739\\
32	0.142541137294759\\
33	0.142584656035631\\
34	0.142454381545442\\
35	0.142773831228876\\
36	0.142146141805815\\
37	0.143765794673104\\
38	0.143740166324445\\
39	0.144495612024717\\
40	0.144495442190171\\
41	0.144761572863615\\
42	0.144496452304871\\
43	0.145159700317502\\
44	0.144966686217811\\
45	0.145850142043123\\
46	0.145790309965857\\
47	0.14629994544053\\
48	0.146084173314383\\
49	0.146172584771254\\
50	0.146581050308765\\
};
\addplot [color=mycolor2, mark=square, mark size=1.5pt,mark options={solid, mycolor2}, forget plot]
  table[row sep=crcr]{%
1	0.618587589953284\\
2	0.373853908171904\\
3	0.284898584515542\\
4	0.226227685220542\\
5	0.224880021765504\\
6	0.174584080001747\\
7	0.165524088997859\\
8	0.164937751094724\\
9	0.102777743304529\\
10	0.0595044482770422\\
11	0.0595339742742175\\
12	0.0578605246744232\\
13	0.0568006052889583\\
14	0.0576664163965346\\
15	0.0584809758413764\\
16	0.0619462838294377\\
17	0.0644864483165485\\
18	0.0672598151396459\\
19	0.0706952643147598\\
20	0.0732979151451628\\
21	0.0827885510291773\\
22	0.0884241942690219\\
23	0.0921214195627612\\
24	0.0960499589413616\\
25	0.0988647413270703\\
26	0.101635212766448\\
27	0.103728427870244\\
28	0.105141404163134\\
29	0.110458229350788\\
30	0.111726909645059\\
31	0.114253150013447\\
32	0.114779935988195\\
33	0.116299512327692\\
34	0.118861909881131\\
35	0.120068313104321\\
36	0.121830224402264\\
37	0.122793526514197\\
38	0.123326003761805\\
39	0.125285501289176\\
40	0.126159961653854\\
41	0.126991395385551\\
42	0.127803458652085\\
43	0.128653969202887\\
44	0.129544949625261\\
45	0.130174448829532\\
46	0.130782933151189\\
47	0.131425614195621\\
48	0.131914544287008\\
49	0.132243493224839\\
50	0.133178027906657\\
};

\addplot [color=mycolor4, mark=star, mark size=1.5pt,mark options={solid, mycolor4}, forget plot]
  table[row sep=crcr]{%
1	0.618587589953284\\
2	0.373853908171904\\
3	0.284898584515542\\
4	0.226227685220542\\
5	0.223758845445071\\
6	0.174584080001747\\
7	0.165524088997859\\
8	0.167991372764148\\
9	0.102725534261395\\
10	0.0674890389338511\\
11	0.0643362974272851\\
12	0.0641767178567442\\
13	0.066294978761119\\
14	0.0646084646918679\\
15	0.0651325896449415\\
16	0.0666267285552637\\
17	0.0673975262995887\\
18	0.0662814672942024\\
19	0.0679688742695086\\
20	0.0734876758616741\\
21	0.079031567239592\\
22	0.0819630783546615\\
23	0.0817822158868105\\
24	0.0892194786795643\\
25	0.0895886368867561\\
26	0.0943728560524134\\
27	0.0919325429023422\\
28	0.101729858323457\\
29	0.0973806747738048\\
30	0.101815905144833\\
31	0.104067642515477\\
32	0.103933829747912\\
33	0.10449982515256\\
34	0.107368523725019\\
35	0.109611476455835\\
36	0.111340074526143\\
37	0.113069652546404\\
38	0.113904917437643\\
39	0.119590363048604\\
40	0.119543103540091\\
41	0.120147315705114\\
42	0.117937866229693\\
43	0.118851927811945\\
44	0.120933218589676\\
45	0.122489142331537\\
46	0.124619750930282\\
47	0.124458326176003\\
48	0.126538031734515\\
49	0.127829282043038\\
50	0.128841712847919\\
};
\addplot [color=mycolor5, mark=diamond, mark size=1.5pt,mark options={solid, mycolor5}, forget plot]
  table[row sep=crcr]{%
1	0.618499736892814\\
2	0.373735810741103\\
3	0.283999535789841\\
4	0.224951148806723\\
5	0.190929949136154\\
6	0.171711434714284\\
7	0.157258734780401\\
8	0.140864401028916\\
9	0.131984720058802\\
10	0.14963467252609\\
11	0.149807195400161\\
12	0.149872266128427\\
13	0.149943710520612\\
14	0.149964818042608\\
15	0.14996704956764\\
16	0.149967277963146\\
17	0.149968537831492\\
18	0.149970531118461\\
19	0.149972340964272\\
20	0.14998330724411\\
21	0.149987328338669\\
22	0.149988655574895\\
23	0.149990219098356\\
24	0.149990905935946\\
25	0.149990958434226\\
26	0.149991857001988\\
27	0.149993136574355\\
28	0.149993650869998\\
29	0.149994534753426\\
30	0.149995299733663\\
31	0.149995760280893\\
32	0.149995601542509\\
33	0.149996229743861\\
34	0.149996869000796\\
35	0.149997229444971\\
36	0.149997561449742\\
37	0.14999801127511\\
38	0.149998404964211\\
39	0.149998373146979\\
40	0.149998355535438\\
41	0.14999863628408\\
42	0.149998812911562\\
43	0.149998883299499\\
44	0.149998970985403\\
45	0.149999034294212\\
46	0.149999067011759\\
47	0.149999135744737\\
48	0.14999937514744\\
49	0.149999483161381\\
50	0.149999546574631\\
};
\addplot [color=mycolor6, mark=+, mark size=1.5pt,mark options={solid, mycolor6}, forget plot]
  table[row sep=crcr]{%
1	0.619089083844624\\
2	0.373758107379424\\
3	0.285591784626085\\
4	0.227018461599264\\
5	0.191327952238398\\
6	0.173772649765011\\
7	0.15770673006643\\
8	0.141930733700263\\
9	0.101577035861298\\
10	0.0503437908735805\\
11	0.04930336345861\\
12	0.0434680647082711\\
13	0.0390637387656002\\
14	0.0365342996858554\\
15	0.0352237746041344\\
16	0.0341839605544175\\
17	0.031029178121289\\
18	0.0310206255882332\\
19	0.0300792570820822\\
20	0.0262483353459674\\
21	0.0250143991672615\\
22	0.0237268425540142\\
23	0.0231119995380747\\
24	0.0222783519099486\\
25	0.0222622795255874\\
26	0.0220700302897558\\
27	0.0214642877996989\\
28	0.0202564032590593\\
29	0.0194750640953821\\
30	0.0187391003149348\\
31	0.018129001232857\\
32	0.017186971842651\\
33	0.0168766311902077\\
34	0.0163199590522879\\
35	0.0156975664634151\\
36	0.0151883262337868\\
37	0.015078455564657\\
38	0.0147510010642847\\
39	0.0146423231854003\\
40	0.0142339500935338\\
41	0.0140018205424998\\
42	0.0134458094363076\\
43	0.0126014093202007\\
44	0.0123657589764408\\
45	0.0114533884274047\\
46	0.0110835772914783\\
47	0.0107974400302043\\
48	0.0107582948215377\\
49	0.00989686840959222\\
50	0.00772849707097023\\
};
\end{axis}

\end{tikzpicture}%

%% file: Plots/RSVD_CUR_plots/ex1_0.1.tex
%
%
\definecolor{mycolor1}{rgb}{0.00000,0.44700,0.74100}%
\definecolor{mycolor2}{rgb}{0.85000,0.32500,0.09800}%
\definecolor{mycolor3}{rgb}{0.92900,0.69400,0.12500}%
\definecolor{mycolor4}{rgb}{0.49400,0.18400,0.55600}%
\definecolor{mycolor5}{rgb}{0.46600,0.67400,0.18800}%
\definecolor{mycolor6}{rgb}{0.30100,0.74500,0.93300}%
\begin{tikzpicture}

\begin{axis}[%
width=0.4\textwidth,
height=3.5cm,
at={(0\textwidth,0\textwidth)},
scale only axis,
xmin=1,
xmax=50,
ymin=0,
ymax=0.25,
xlabel=rank $k$,
ylabel={$\| A - \widetilde{A}_k \|\ /\ \| A \|$},
axis background/.style={fill=white},
legend pos=outer north east
]
\addplot [color=blue, mark=o,mark size=1.5pt, mark options={solid, blue}]
  table[row sep=crcr]{%
1	0.61858570775142\\
2	0.373838738849578\\
3	0.284868183019796\\
4	0.240252708466364\\
5	0.223788548495257\\
6	0.180830845631842\\
7	0.17435864369073\\
8	0.173915626590947\\
9	0.10264077864004\\
10	0.0975663110692242\\
11	0.0741368517659358\\
12	0.0772429502278852\\
13	0.0791694376066318\\
14	0.0799782320223976\\
15	0.0816110517926237\\
16	0.0819699573469209\\
17	0.0826333481598634\\
18	0.0837088311980197\\
19	0.0857053542181337\\
20	0.0866623686696185\\
21	0.0868686943762482\\
22	0.087893762884346\\
23	0.0888878628666752\\
24	0.0890063195083511\\
25	0.0904263913751607\\
26	0.0912924899562383\\
27	0.0916473240047549\\
28	0.0924583576657456\\
29	0.0928093499486115\\
30	0.0930897764718062\\
31	0.0933437402377569\\
32	0.0934978340518904\\
33	0.0938496849577055\\
34	0.094064056176599\\
35	0.0943459284320266\\
36	0.0945869860887702\\
37	0.0947600786598649\\
38	0.094967739713799\\
39	0.0952177110197171\\
40	0.0954316820524645\\
41	0.0956612614285679\\
42	0.0958511322454847\\
43	0.0959942175248549\\
44	0.0961214971800699\\
45	0.0963974529827224\\
46	0.0965511483591493\\
47	0.0966307181249683\\
48	0.0967757628830771\\
49	0.0968322063356569\\
50	0.0969172055885621\\
};
\addlegendentry{DEIM-CUR}

\addplot [color=red, mark=asterisk,mark size=1.5pt,  mark options={solid, red}]
  table[row sep=crcr]{%
1	0.61858570775142\\
2	0.373847650940691\\
3	0.284875474509559\\
4	0.226185877155832\\
5	0.195692227722263\\
6	0.174570470986077\\
7	0.165379297798485\\
8	0.16476557716389\\
9	0.102702416209809\\
10	0.0862836646153813\\
11	0.0729743385491812\\
12	0.0737775108768682\\
13	0.0776743576444861\\
14	0.0803981572514816\\
15	0.0818340527696041\\
16	0.0830438307334446\\
17	0.0843053659407953\\
18	0.0858186131708653\\
19	0.0847361096856937\\
20	0.0875071653700941\\
21	0.088758846376341\\
22	0.0883073095324763\\
23	0.0890347280901511\\
24	0.0875306436998498\\
25	0.0889380582059765\\
26	0.089512627088823\\
27	0.0905532290281432\\
28	0.0906194560045435\\
29	0.0914589202387378\\
30	0.0905371572836109\\
31	0.0913686056307103\\
32	0.0923909759775609\\
33	0.0925454734035652\\
34	0.0925626309829367\\
35	0.0936213507798097\\
36	0.0942322590666475\\
37	0.0942503550123126\\
38	0.0944697125350367\\
39	0.0944293771970815\\
40	0.0944988063221427\\
41	0.0951631026780533\\
42	0.0955636346005061\\
43	0.0955538351430543\\
44	0.0962644399228854\\
45	0.0956435687506131\\
46	0.0959053435039445\\
47	0.0963293382799722\\
48	0.0964115992577201\\
49	0.0963394579757136\\
50	0.0963057164567655\\
};
\addlegendentry{QDEIM-CUR}

\addplot [color=mycolor2, mark=square,mark size=1.5pt, mark options={solid, mycolor2}]
  table[row sep=crcr]{%
1	0.61858570775142\\
2	0.373847650940691\\
3	0.284887597457364\\
4	0.226214802585625\\
5	0.224898457664488\\
6	0.174524613757652\\
7	0.165327856819534\\
8	0.164732081420906\\
9	0.102693177723829\\
10	0.053452301473139\\
11	0.0534860480075423\\
12	0.0474775243181702\\
13	0.0462835233882109\\
14	0.0462673670798101\\
15	0.0460772379595614\\
16	0.0456424618912566\\
17	0.0442507791363077\\
18	0.0434402498628182\\
19	0.0447638413196049\\
20	0.0453808567237216\\
21	0.0482012614007792\\
22	0.0507795526222463\\
23	0.0526145921510302\\
24	0.054817180866232\\
25	0.0564826481520757\\
26	0.0581302865202639\\
27	0.0593337700442001\\
28	0.0604082544995824\\
29	0.0630691066152634\\
30	0.0639329636300992\\
31	0.065627966207763\\
32	0.0659950274667779\\
33	0.0670029027965841\\
34	0.0687887836412554\\
35	0.0703390826878223\\
36	0.0710673119308185\\
37	0.0720553976972604\\
38	0.0727862403928456\\
39	0.0747735108875343\\
40	0.0758072866357835\\
41	0.0766567578067809\\
42	0.0772363816301812\\
43	0.0779339710087449\\
44	0.0784842173855519\\
45	0.079315601561266\\
46	0.0801262813313331\\
47	0.0805878972878892\\
48	0.081085724935715\\
49	0.0814130834077264\\
50	0.0822694031059791\\
};
\addlegendentry{DEIM-RSVD-CUR}

\addplot [color=mycolor4, mark=star, mark size=1.5pt, mark options={solid, mycolor4}]
  table[row sep=crcr]{%
1	0.61858570775142\\
2	0.373847650940691\\
3	0.284887597457364\\
4	0.226214802585625\\
5	0.223793971821218\\
6	0.174524613757652\\
7	0.165327856819534\\
8	0.164710224658145\\
9	0.102636225576063\\
10	0.0575711814310072\\
11	0.0555212460946976\\
12	0.0539640957623276\\
13	0.0557742042834762\\
14	0.0555800345468977\\
15	0.0554529788549028\\
16	0.0549160737092532\\
17	0.0548149767403543\\
18	0.055357439793813\\
19	0.0552581718965228\\
20	0.0526943147676941\\
21	0.0498072544842477\\
22	0.0489825489018741\\
23	0.0465354543747045\\
24	0.0499600097762124\\
25	0.0509133334077786\\
26	0.0527624379516493\\
27	0.0524096874539975\\
28	0.057138414643649\\
29	0.0554109300389461\\
30	0.0586395466511053\\
31	0.0599991759886941\\
32	0.059289455274514\\
33	0.0610312442026065\\
34	0.0622517926123453\\
35	0.0629636207099854\\
36	0.0649355134946004\\
37	0.0639837772262489\\
38	0.0677002966763589\\
39	0.0703954129723407\\
40	0.0714915806414178\\
41	0.0722458001628722\\
42	0.0696247056369755\\
43	0.0709861050470607\\
44	0.0723255083503375\\
45	0.0731598219109929\\
46	0.0769660335346884\\
47	0.0759546997104748\\
48	0.0771239440675273\\
49	0.0784230563906222\\
50	0.0798036597101144\\
};
\addlegendentry{QDEIM-RSVD-CUR}

\addplot [color=mycolor5, mark=diamond, mark size=1.5pt, mark options={solid, mycolor5}]
  table[row sep=crcr]{%
1	0.618499601501808\\
2	0.373734892538762\\
3	0.283995544802594\\
4	0.224939278320009\\
5	0.190900814721668\\
6	0.171658846855349\\
7	0.157070497518413\\
8	0.140600841369687\\
9	0.101520914021384\\
10	0.055793661011709\\
11	0.0998400984829982\\
12	0.0998490805417751\\
13	0.0998740144385101\\
14	0.0998764624750204\\
15	0.0998928502737389\\
16	0.0998954123165885\\
17	0.0999302415255217\\
18	0.0999471768238345\\
19	0.0999544390790734\\
20	0.0999603394309056\\
21	0.0999643917887586\\
22	0.0999667622426858\\
23	0.0999688042743287\\
24	0.0999730555666229\\
25	0.0999787483686031\\
26	0.0999825415966006\\
27	0.0999847780853258\\
28	0.0999858931966525\\
29	0.0999873905802988\\
30	0.0999891957623766\\
31	0.099990329410315\\
32	0.099992012711888\\
33	0.099992435715872\\
34	0.0999937150081774\\
35	0.0999943176717474\\
36	0.099994690818476\\
37	0.0999952205638917\\
38	0.0999956648901051\\
39	0.0999963095171315\\
40	0.0999972620529281\\
41	0.0999979320471552\\
42	0.0999982113790842\\
43	0.099998193821655\\
44	0.0999984902281839\\
45	0.0999985377573702\\
46	0.0999986657826672\\
47	0.0999986707798501\\
48	0.0999987924609163\\
49	0.0999988272136873\\
50	0.0999990867385357\\
};
\addlegendentry{SVD}

\addplot [color=mycolor6, mark=+, mark size=1.5pt, mark options={solid, mycolor6}]
  table[row sep=crcr]{%
1	0.61908908776182\\
2	0.373758099054756\\
3	0.285591682795714\\
4	0.227018029487604\\
5	0.191327497148098\\
6	0.173772716940304\\
7	0.157705947000763\\
8	0.141929816252313\\
9	0.101576620030575\\
10	0.0503357459045396\\
11	0.049297168048903\\
12	0.0434623329755435\\
13	0.0390603468148501\\
14	0.0365294583917007\\
15	0.0352034219766175\\
16	0.0341615797951084\\
17	0.0309923272960697\\
18	0.0309830951027361\\
19	0.0300355357189081\\
20	0.0262478744340724\\
21	0.024998939893539\\
22	0.0237106443929794\\
23	0.0230652458380812\\
24	0.0222539334582003\\
25	0.0222367482353292\\
26	0.0220363005951427\\
27	0.0214390946861313\\
28	0.020242094425505\\
29	0.0194704444735152\\
30	0.0186973244116094\\
31	0.0180868794229731\\
32	0.0171816301021681\\
33	0.0168520351967566\\
34	0.0163154962150332\\
35	0.0156918382643458\\
36	0.0151839331235311\\
37	0.015054921246237\\
38	0.0147225489259903\\
39	0.0146079336797585\\
40	0.0142256942829856\\
41	0.0139914092648215\\
42	0.0134228587719296\\
43	0.0125884383984018\\
44	0.0123572257118836\\
45	0.0114446655379605\\
46	0.0110685667732108\\
47	0.0107733252756026\\
48	0.010741511317375\\
49	0.00987964069237213\\
50	0.00514495760189489\\
};
\addlegendentry{RSVD}

\end{axis}

\end{tikzpicture}%

%% file: Plots/RSVD_CUR_plots/boundn.tex
%
%
\begin{tikzpicture}

\begin{axis}[%
width=0.4\textwidth,
height=3.5cm,
at={(0cm,0cm)},
scale only axis,
xmin=1,
xmax=50,
ymode=log,
ymin=1,
ymax=10000000,
yminorticks=true,
axis background/.style={fill=white},
legend pos=outer north east
]
\addplot [color=red]
  table[row sep=crcr]{%
1	9.03541507786725\\
2	9.05461637056752\\
3	9.26328349786734\\
4	9.29443159383141\\
5	9.60654261353564\\
6	10.4937166685237\\
7	11.78457653222\\
8	11.1357759473105\\
9	11.1646295953123\\
10	17.3592016636258\\
11	17.3609328927901\\
12	17.3531360379414\\
13	17.3381365381207\\
14	17.7381860726081\\
15	17.7397281699735\\
16	17.8959332581668\\
17	24.1117400919365\\
18	23.014652989373\\
19	23.5098414317779\\
20	23.6220973499531\\
21	24.0466498278443\\
22	23.0393558769694\\
23	23.0093515042351\\
24	23.0600100771654\\
25	23.5920787045176\\
26	24.6773005729898\\
27	25.36790486331\\
28	24.2395079566936\\
29	25.4230046462976\\
30	25.7899315618679\\
31	24.7832589375138\\
32	26.0037444709257\\
33	26.3251600255227\\
34	28.0357334385579\\
35	29.4030249993699\\
36	30.8043937648428\\
37	28.0873214164272\\
38	28.736635241071\\
39	29.1764996918394\\
40	28.1662501687267\\
41	31.6285388773565\\
42	33.1020040739779\\
43	30.6050528286326\\
44	33.8651489589827\\
45	34.6907182952191\\
46	35.9859507463175\\
47	37.8564943148798\\
48	41.2708419729982\\
49	41.2023779689768\\
50	40.378707547162\\
};
\addlegendentry{$\eta_s$}

\addplot [color=red, dashed]
  table[row sep=crcr]{%
1	3.09272266802388\\
2	3.09252422308324\\
3	3.13666484569414\\
4	4.33540412069492\\
5	4.33582047703507\\
6	4.33771438102516\\
7	4.56531744588117\\
8	6.28105564191516\\
9	6.29408408113346\\
10	6.26807915062694\\
11	6.2684423867477\\
12	7.3569219579499\\
13	7.4183132269556\\
14	7.38367283561077\\
15	7.38769901702016\\
16	7.45209373405929\\
17	7.47343175753972\\
18	7.52234027889866\\
19	7.49937716228286\\
20	7.51252344618844\\
21	7.58953124929045\\
22	7.14837672769272\\
23	7.19179055972869\\
24	7.39386121956852\\
25	7.40559881169929\\
26	7.46255737435728\\
27	8.19502587267436\\
28	8.40228144411053\\
29	8.28484048777967\\
30	7.20664721174371\\
31	8.29775121763298\\
32	8.65300055101368\\
33	8.54577720736056\\
34	9.54249557535596\\
35	8.71183364916172\\
36	13.6919679791898\\
37	13.9141905597566\\
38	13.5751246196672\\
39	14.4613503309014\\
40	17.9755609135494\\
41	15.2126737323388\\
42	13.7971520325192\\
43	14.1857213487862\\
44	16.081876938523\\
45	15.3246791915802\\
46	15.14893795899\\
47	15.646090184012\\
48	17.4476661613164\\
49	16.7058244480964\\
50	17.5760318517221\\
};
\addlegendentry{$\eta_p$}

\addplot [color=green]
  table[row sep=crcr]{%
1	230.807937756524\\
2	174.254853667682\\
3	135.587095202912\\
4	113.195159409232\\
5	100.271346933154\\
6	100.159665297562\\
7	100.141937205873\\
8	100.141000025848\\
9	100.13607607536\\
10	100.132985105315\\
11	100.131774900758\\
12	100.130721127037\\
13	100.129887830569\\
14	100.129594041724\\
15	100.129045803062\\
16	100.128801803383\\
17	100.12657271606\\
18	100.125140264115\\
19	100.122184408114\\
20	100.122091718348\\
21	100.121724880743\\
22	100.121606461148\\
23	100.120523689046\\
24	100.120420656162\\
25	100.120379021347\\
26	100.120184011627\\
27	100.119285534991\\
28	100.119110621353\\
29	100.118854096229\\
30	100.118778435932\\
31	100.118721117106\\
32	100.118602383529\\
33	100.118154501398\\
34	100.117965041318\\
35	100.117264485255\\
36	100.117062225346\\
37	100.116741226032\\
38	100.116523707934\\
39	100.116361656985\\
40	100.116340546372\\
41	100.116226368608\\
42	100.11601599126\\
43	100.115692505212\\
44	100.115577137834\\
45	100.115521348377\\
46	100.115345396955\\
47	100.115289638314\\
48	100.115203624859\\
49	100.115011355288\\
50	100.114913317238\\
};
\addlegendentry{$\|\widehat T_Z\|$}

\addplot [color=green, dashed]
  table[row sep=crcr]{%
1	107.777016597459\\
2	107.777014753865\\
3	107.777012166043\\
4	107.777009847876\\
5	107.777007873043\\
6	107.777006084041\\
7	107.777004882915\\
8	107.777003533964\\
9	107.776999031069\\
10	107.776998181697\\
11	107.776994419055\\
12	107.776991162511\\
13	107.776986019641\\
14	107.776985644043\\
15	107.776983220794\\
16	107.77697977169\\
17	107.776979751741\\
18	107.776978730457\\
19	107.77696430155\\
20	107.7769639936\\
21	107.776952341181\\
22	107.776951788024\\
23	107.776939311227\\
24	107.776938994915\\
25	107.776935840285\\
26	107.776935769054\\
27	107.776923586755\\
28	107.776923234796\\
29	107.776921051401\\
30	107.776920516133\\
31	107.776914228841\\
32	107.776912858163\\
33	107.776910070969\\
34	107.776908993577\\
35	107.776905399664\\
36	107.776901063779\\
37	107.776900839166\\
38	107.776899511912\\
39	107.776894551159\\
40	107.77689443986\\
41	107.776889860577\\
42	107.776889518247\\
43	107.776889251044\\
44	107.776889055292\\
45	107.776883133845\\
46	107.776878777092\\
47	107.776869489792\\
48	107.776869235258\\
49	107.776868919298\\
50	107.776868067532\\
};
\addlegendentry{$\|\widehat T_W\|$}

\addplot [color=blue]
  table[row sep=crcr]{%
1	23.2938746666469\\
2	17.6481182342788\\
3	13.6635464136767\\
4	11.5214527310042\\
5	9.05683101272088\\
6	8.40048813364201\\
7	7.96343276328016\\
8	5.79333926273529\\
9	5.87221194386794\\
10	5.67686703640936\\
11	5.59867778879269\\
12	5.62944969149224\\
13	5.58839844803164\\
14	5.52877675825644\\
15	5.40950183473849\\
16	5.32030472721403\\
17	5.27998287320719\\
18	5.22528252472548\\
19	5.17242575003505\\
20	5.11880024587725\\
21	5.075045043237\\
22	5.02470666693318\\
23	4.94738608433083\\
24	4.8249577293484\\
25	4.7879022139963\\
26	4.6664238179076\\
27	4.5579817054869\\
28	4.45077820799392\\
29	4.41577148648997\\
30	4.33815929176842\\
31	4.32470249222876\\
32	4.19640610070407\\
33	4.13992394157828\\
34	4.11393582872202\\
35	4.05012362009059\\
36	4.00369334437812\\
37	3.97995119523163\\
38	3.89999235696657\\
39	3.8425392250639\\
40	3.7770515687245\\
41	3.76204008769297\\
42	3.65439391937404\\
43	3.562302229644\\
44	3.5406525847841\\
45	3.48257475357689\\
46	3.44015068253322\\
47	3.40193096164985\\
48	3.36378348547959\\
49	3.33095444490541\\
50	3.29635684376472\\
};
\addlegendentry{RSVD-CUR-error}

\addplot [color=blue, dashed]
  table[row sep=crcr]{%
1	213320.669554249\\
2	161297.029968249\\
3	128108.980741869\\
4	117556.16154981\\
5	106490.207495615\\
6	113141.314355575\\
7	124679.76931521\\
8	132755.487797698\\
9	133057.846081983\\
10	179613.058482285\\
11	179463.908845694\\
12	187450.154151197\\
13	187712.233328817\\
14	189844.72079038\\
15	189573.030531374\\
16	191070.035339867\\
17	237689.090918842\\
18	229647.271070982\\
19	233011.642565758\\
20	233303.192012668\\
21	236556.634032689\\
22	224171.939893922\\
23	223839.250829526\\
24	224922.426603714\\
25	228556.503117558\\
26	236692.986808391\\
27	247051.139885166\\
28	239366.480760127\\
29	247401.304446132\\
30	241836.48834098\\
31	240878.845894972\\
32	252014.463658265\\
33	252888.991613233\\
34	272396.14746606\\
35	277466.74203743\\
36	322069.086177496\\
37	301944.422743332\\
38	303208.749464866\\
39	311744.41032502\\
40	328181.140331544\\
41	330519.090139022\\
42	332045.238479883\\
43	315968.007264671\\
44	350646.25001155\\
45	349509.897451734\\
46	353502.781583047\\
47	369536.144233698\\
48	404093.512545202\\
49	398976.835235659\\
50	387183.519468177\\
};
\addlegendentry{RSVD-CUR-bound}

\addplot [color=black]
  table[row sep=crcr]{%
1	54.1657258751031\\
2	54.1657258751031\\
3	54.1657258751031\\
4	54.1657258751031\\
5	54.1657258751031\\
6	54.1657258751031\\
7	54.1657258751031\\
8	54.1657258751031\\
9	54.1657258751031\\
10	54.1657258751031\\
11	54.1657258751031\\
12	54.1657258751031\\
13	54.1657258751031\\
14	54.1657258751031\\
15	54.1657258751031\\
16	54.1657258751031\\
17	54.1657258751031\\
18	54.1657258751031\\
19	54.1657258751031\\
20	54.1657258751031\\
21	54.1657258751031\\
22	54.1657258751031\\
23	54.1657258751031\\
24	54.1657258751031\\
25	54.1657258751031\\
26	54.1657258751031\\
27	54.1657258751031\\
28	54.1657258751031\\
29	54.1657258751031\\
30	54.1657258751031\\
31	54.1657258751031\\
32	54.1657258751031\\
33	54.1657258751031\\
34	54.1657258751031\\
35	54.1657258751031\\
36	54.1657258751031\\
37	54.1657258751031\\
38	54.1657258751031\\
39	54.1657258751031\\
40	54.1657258751031\\
41	54.1657258751031\\
42	54.1657258751031\\
43	54.1657258751031\\
44	54.1657258751031\\
45	54.1657258751031\\
46	54.1657258751031\\
47	54.1657258751031\\
48	54.1657258751031\\
49	54.1657258751031\\
50	54.1657258751031\\
};
\addlegendentry{$\|A_E\|$}

\end{axis}
\end{tikzpicture}%

%% file: Plots/RSVD_CUR_plots/inex_0.1.tex
%
%
\definecolor{mycolor1}{rgb}{0.00000,0.44700,0.74100}%
\definecolor{mycolor2}{rgb}{0.85000,0.32500,0.09800}%
\definecolor{mycolor3}{rgb}{0.92900,0.69400,0.12500}%
\definecolor{mycolor4}{rgb}{0.49400,0.18400,0.55600}%
\definecolor{mycolor5}{rgb}{0.46600,0.67400,0.18800}%
\definecolor{mycolor6}{rgb}{0.30100,0.74500,0.93300}%
\begin{tikzpicture}

\begin{axis}[%
width=0.4\textwidth,
height=3.5cm,
at={(0\textwidth,0\textwidth)},
scale only axis,
xmin=1,
xmax=50,
ymin=0,
ymax=0.25,
xlabel=rank $k$,
ylabel={$\| A - \widetilde{A}_k \|\ /\ \| A \|$},
axis background/.style={fill=white},
legend style={legend columns=1, legend cell align=left, align=left, draw=white!15!black,nodes={scale=0.8, transform shape}}
]
\addplot [color=blue, mark=o,  mark size=1.5pt,mark options={solid, blue}]
  table[row sep=crcr]{%
1	0.61858570775142\\
2	0.373838738849578\\
3	0.284868183019796\\
4	0.240252708466364\\
5	0.223788548495257\\
6	0.180830845631842\\
7	0.17435864369073\\
8	0.173915626590947\\
9	0.10264077864004\\
10	0.0975663110692242\\
11	0.0741368517659358\\
12	0.0772429502278852\\
13	0.0791694376066318\\
14	0.0799782320223976\\
15	0.0816110517926237\\
16	0.0819699573469209\\
17	0.0826333481598634\\
18	0.0837088311980197\\
19	0.0857053542181337\\
20	0.0866623686696185\\
21	0.0868686943762482\\
22	0.087893762884346\\
23	0.0888878628666752\\
24	0.0890063195083511\\
25	0.0904263913751607\\
26	0.0912924899562383\\
27	0.0916473240047549\\
28	0.0924583576657456\\
29	0.0928093499486115\\
30	0.0930897764718062\\
31	0.0933437402377569\\
32	0.0934978340518904\\
33	0.0938496849577055\\
34	0.094064056176599\\
35	0.0943459284320266\\
36	0.0945869860887702\\
37	0.0947600786598649\\
38	0.094967739713799\\
39	0.0952177110197171\\
40	0.0954316820524645\\
41	0.0956612614285679\\
42	0.0958511322454847\\
43	0.0959942175248549\\
44	0.0961214971800699\\
45	0.0963974529827224\\
46	0.0965511483591493\\
47	0.0966307181249683\\
48	0.0967757628830771\\
49	0.0968322063356569\\
50	0.0969172055885621\\
};
\addlegendentry{DEIM-CUR}

\addplot [color=red, mark=asterisk,  mark size=1.5pt,mark options={solid, red}]
  table[row sep=crcr]{%
1	0.61858570775142\\
2	0.373847650940691\\
3	0.284875474509559\\
4	0.226185877155832\\
5	0.195692227722263\\
6	0.174570470986077\\
7	0.165379297798485\\
8	0.16476557716389\\
9	0.102702416209809\\
10	0.0862836646153813\\
11	0.0729743385491812\\
12	0.0737775108768682\\
13	0.0776743576444861\\
14	0.0803981572514816\\
15	0.0818340527696041\\
16	0.0830438307334446\\
17	0.0843053659407953\\
18	0.0858186131708653\\
19	0.0847361096856937\\
20	0.0875071653700941\\
21	0.088758846376341\\
22	0.0883073095324763\\
23	0.0890347280901511\\
24	0.0875306436998498\\
25	0.0889380582059765\\
26	0.089512627088823\\
27	0.0905532290281432\\
28	0.0906194560045435\\
29	0.0914589202387378\\
30	0.0905371572836109\\
31	0.0913686056307103\\
32	0.0923909759775609\\
33	0.0925454734035652\\
34	0.0925626309829367\\
35	0.0936213507798097\\
36	0.0942322590666475\\
37	0.0942503550123126\\
38	0.0944697125350367\\
39	0.0944293771970815\\
40	0.0944988063221427\\
41	0.0951631026780533\\
42	0.0955636346005061\\
43	0.0955538351430543\\
44	0.0962644399228854\\
45	0.0956435687506131\\
46	0.0959053435039445\\
47	0.0963293382799722\\
48	0.0964115992577201\\
49	0.0963394579757136\\
50	0.0963057164567655\\
};
\addlegendentry{QDEIM-CUR}

\addplot [color=mycolor2, mark=square, mark size=1.5pt, mark options={solid, mycolor2}]
  table[row sep=crcr]{%
1	0.61858570775142\\
2	0.373847650940691\\
3	0.284887597457364\\
4	0.240284746698447\\
5	0.223793971821218\\
6	0.174516604843976\\
7	0.165307096673973\\
8	0.164725944593043\\
9	0.102692835849739\\
10	0.0535319478303391\\
11	0.0534719034190731\\
12	0.0504516656608384\\
13	0.0500550495326531\\
14	0.0499995133897044\\
15	0.0475734013862294\\
16	0.0468066212497584\\
17	0.0460496381481802\\
18	0.0447271389113012\\
19	0.0465774163561713\\
20	0.0467241592003913\\
21	0.0487773727559313\\
22	0.0506388545834017\\
23	0.0510812467053694\\
24	0.052224523571123\\
25	0.0529594300288153\\
26	0.0553217492071842\\
27	0.0559132545100707\\
28	0.0575133418500634\\
29	0.0585712417243561\\
30	0.0608187076774207\\
31	0.0618991799569113\\
32	0.0628116062724637\\
33	0.0635253031905997\\
34	0.0647517990575684\\
35	0.0661176681832179\\
36	0.066875304029272\\
37	0.0682677889743648\\
38	0.068651125005725\\
39	0.070332436225517\\
40	0.0717723990626666\\
41	0.0729084046473922\\
42	0.074087961669079\\
43	0.0748924391427304\\
44	0.0755802401040303\\
45	0.0763607153791759\\
46	0.0769702487684608\\
47	0.0783778275324297\\
48	0.0791209191732683\\
49	0.0795990128952378\\
50	0.0806010467139267\\
};
\addlegendentry{DEIM-RSVD-CUR}

\addplot [color=mycolor4, mark=star,  mark size=1.5pt,mark options={solid, mycolor4}]
  table[row sep=crcr]{%
1	0.61858570775142\\
2	0.373847650940691\\
3	0.284887597457364\\
4	0.240284746698447\\
5	0.223793971821218\\
6	0.174516604843976\\
7	0.165330212162327\\
8	0.164704829854952\\
9	0.102640979374598\\
10	0.0534699664489024\\
11	0.0534166195164022\\
12	0.0479857527287828\\
13	0.0464339791091431\\
14	0.0469846216221341\\
15	0.045459868193408\\
16	0.0455085583722799\\
17	0.0451712219567233\\
18	0.0421627142858095\\
19	0.0417069364549168\\
20	0.0432284736310688\\
21	0.0444904727681299\\
22	0.045701693996672\\
23	0.0470728841379334\\
24	0.0497536277720677\\
25	0.0475335483948985\\
26	0.0530421818110763\\
27	0.0529368905738373\\
28	0.0530985576826258\\
29	0.0545744827250575\\
30	0.058226126388537\\
31	0.0568810584912283\\
32	0.0590333239505682\\
33	0.059419897535508\\
34	0.0607692578667013\\
35	0.0631508008582417\\
36	0.0634697740922599\\
37	0.0644189534808016\\
38	0.0661146555918232\\
39	0.0692969377437091\\
40	0.0698629538467231\\
41	0.0690200722976487\\
42	0.0709659835042458\\
43	0.0709101287926426\\
44	0.0731204963521562\\
45	0.0745944222082032\\
46	0.0761369944299772\\
47	0.0799847260700109\\
48	0.0789116195887686\\
49	0.0790618997295201\\
50	0.0801773085239873\\
};
\addlegendentry{QDEIM-RSVD-CUR}

\addplot [color=mycolor5, mark=diamond,  mark size=1.5pt,mark options={solid, mycolor5}]
  table[row sep=crcr]{%
1	0.618499601501808\\
2	0.373734892538762\\
3	0.283995544802594\\
4	0.224939278320009\\
5	0.190900814721668\\
6	0.171658846855349\\
7	0.157070497518413\\
8	0.140600841369687\\
9	0.101520914021384\\
10	0.055793661011709\\
11	0.0998400984829982\\
12	0.0998490805417751\\
13	0.0998740144385101\\
14	0.0998764624750204\\
15	0.0998928502737389\\
16	0.0998954123165885\\
17	0.0999302415255217\\
18	0.0999471768238345\\
19	0.0999544390790734\\
20	0.0999603394309056\\
21	0.0999643917887586\\
22	0.0999667622426858\\
23	0.0999688042743287\\
24	0.0999730555666229\\
25	0.0999787483686031\\
26	0.0999825415966006\\
27	0.0999847780853258\\
28	0.0999858931966525\\
29	0.0999873905802988\\
30	0.0999891957623766\\
31	0.099990329410315\\
32	0.099992012711888\\
33	0.099992435715872\\
34	0.0999937150081774\\
35	0.0999943176717474\\
36	0.099994690818476\\
37	0.0999952205638917\\
38	0.0999956648901051\\
39	0.0999963095171315\\
40	0.0999972620529281\\
41	0.0999979320471552\\
42	0.0999982113790842\\
43	0.099998193821655\\
44	0.0999984902281839\\
45	0.0999985377573702\\
46	0.0999986657826672\\
47	0.0999986707798501\\
48	0.0999987924609163\\
49	0.0999988272136873\\
50	0.0999990867385357\\
};
\addlegendentry{SVD}

\addplot [color=mycolor6, mark=+, mark size=1.5pt,mark options={solid, mycolor6}]
  table[row sep=crcr]{%
1	0.619867266860878\\
2	0.373870237448787\\
3	0.284622245624577\\
4	0.228357607699957\\
5	0.191794035940755\\
6	0.173377050794007\\
7	0.160349734716035\\
8	0.144106232192093\\
9	0.102008318167329\\
10	0.0500491698630396\\
11	0.0486737270262388\\
12	0.0440913259025569\\
13	0.0404898683060642\\
14	0.0377455240480731\\
15	0.0349335205223625\\
16	0.0335999471260295\\
17	0.0312251591008229\\
18	0.0311835877925997\\
19	0.0281721691345202\\
20	0.0268041515163745\\
21	0.0265159211827571\\
22	0.0246512265372866\\
23	0.0239137137206792\\
24	0.0239149293515401\\
25	0.0226117721117501\\
26	0.0222000041435325\\
27	0.0216459177316402\\
28	0.0216074239185029\\
29	0.0203615574045115\\
30	0.0192340675722644\\
31	0.0185017943857191\\
32	0.0184283690556957\\
33	0.017376248489079\\
34	0.0170077261458653\\
35	0.0168592027291692\\
36	0.0156141833601353\\
37	0.0156005571530469\\
38	0.015561482485848\\
39	0.0155542828779956\\
40	0.0153005487462799\\
41	0.0152669770658383\\
42	0.014099262877178\\
43	0.0131613762498589\\
44	0.0131128705219017\\
45	0.0125780365408205\\
46	0.0120379310621282\\
47	0.0112940600659963\\
48	0.0112876180904979\\
49	0.0101871467588145\\
50	0.00526921470439507\\
};
\addlegendentry{RSVD}

\end{axis}

\end{tikzpicture}%

%% file: Plots/RSVD_CUR_plots/inex_0.2.tex
%
%
\definecolor{mycolor1}{rgb}{0.00000,0.44700,0.74100}%
\definecolor{mycolor2}{rgb}{0.85000,0.32500,0.09800}%
\definecolor{mycolor3}{rgb}{0.92900,0.69400,0.12500}%
\definecolor{mycolor4}{rgb}{0.49400,0.18400,0.55600}%
\definecolor{mycolor5}{rgb}{0.46600,0.67400,0.18800}%
\definecolor{mycolor6}{rgb}{0.30100,0.74500,0.93300}%
\begin{tikzpicture}

\begin{axis}[%
width=0.4\textwidth,
height=3.5cm,
at={(0\textwidth,0\textwidth)},
scale only axis,
xmin=1,
xmax=50,
ymin=0,
ymax=0.25,
xlabel=rank $k$,
axis background/.style={fill=white},
]
\addplot [color=blue, mark=o,  mark size=1.5pt,mark options={solid, blue}, forget plot]
  table[row sep=crcr]{%
1	0.618589639170852\\
2	0.37383795225893\\
3	0.284852787669925\\
4	0.238909243728072\\
5	0.223833607872855\\
6	0.203850677040055\\
7	0.18695656740094\\
8	0.182880149424912\\
9	0.181158163733663\\
10	0.160905415797211\\
11	0.16581108516302\\
12	0.170159193927934\\
13	0.172806002698074\\
14	0.173826396162659\\
15	0.176247051490059\\
16	0.178662427834773\\
17	0.179622936633376\\
18	0.181981889641854\\
19	0.182956360820265\\
20	0.184029940009212\\
21	0.185246068649812\\
22	0.186681084257055\\
23	0.187074851141511\\
24	0.187657297457112\\
25	0.188750142217667\\
26	0.18958197570928\\
27	0.190593745578657\\
28	0.19118173395958\\
29	0.192146110766377\\
30	0.192624074077385\\
31	0.193293722775025\\
32	0.193553082570238\\
33	0.193773953233855\\
34	0.193968877643802\\
35	0.194143346517406\\
36	0.194451237176449\\
37	0.194631481838725\\
38	0.194946349561951\\
39	0.195139361956872\\
40	0.195399690262316\\
41	0.195560330446947\\
42	0.195699402170255\\
43	0.19585860715172\\
44	0.196067904585662\\
45	0.196101381610148\\
46	0.196219934953358\\
47	0.196382360629493\\
48	0.196508980810189\\
49	0.196629560417201\\
50	0.196768685428789\\
};
\addplot [color=red, mark=asterisk, mark size=1.5pt, mark options={solid, red}, forget plot]
  table[row sep=crcr]{%
1	0.618589639170852\\
2	0.373861181336824\\
3	0.284870754356218\\
4	0.232414767665696\\
5	0.216050876297982\\
6	0.196766170838173\\
7	0.176530924723648\\
8	0.172266971547772\\
9	0.161721688031019\\
10	0.157123873247885\\
11	0.163436172579086\\
12	0.16818960851549\\
13	0.168144985972969\\
14	0.173415411516413\\
15	0.172047233646717\\
16	0.173820412335712\\
17	0.182551206466406\\
18	0.178849476829185\\
19	0.182323400715927\\
20	0.182404434512478\\
21	0.182165449086699\\
22	0.186376589986489\\
23	0.18394048874869\\
24	0.184986123621969\\
25	0.187915088136146\\
26	0.186521334266127\\
27	0.1891737224253\\
28	0.189173196422401\\
29	0.190543131250562\\
30	0.191195937802008\\
31	0.189885379519879\\
32	0.191279175708339\\
33	0.19145002396696\\
34	0.192097916623212\\
35	0.191866048233116\\
36	0.192576355119453\\
37	0.192633571878952\\
38	0.193049538622912\\
39	0.193630055988868\\
40	0.194096518172343\\
41	0.194317293085872\\
42	0.19449389659018\\
43	0.194287874822646\\
44	0.194353521446331\\
45	0.194836622444348\\
46	0.195364264692426\\
47	0.195494000967541\\
48	0.195652488499528\\
49	0.196012696211215\\
50	0.196112629289978\\
};
\addplot [color=mycolor2, mark=square,  mark size=1.5pt,mark options={solid, mycolor2}, forget plot]
  table[row sep=crcr]{%
1	0.618589639170852\\
2	0.373861181336824\\
3	0.284912311088637\\
4	0.240294083500818\\
5	0.223717298395481\\
6	0.174680043850606\\
7	0.165804273034759\\
8	0.165237447923286\\
9	0.102959887829783\\
10	0.0729113173051374\\
11	0.0739419558762537\\
12	0.0741792041580303\\
13	0.0743951251341259\\
14	0.0783777701896925\\
15	0.0830422109601645\\
16	0.0898304008181092\\
17	0.0945187422452946\\
18	0.101507355447009\\
19	0.107085852799032\\
20	0.109676230677627\\
21	0.118410508327726\\
22	0.128660187548646\\
23	0.132271813710529\\
24	0.134288969112484\\
25	0.137030707751213\\
26	0.142774488961039\\
27	0.143863092934157\\
28	0.147301519921813\\
29	0.149424238838213\\
30	0.153731445322811\\
31	0.155100141178444\\
32	0.156190517777644\\
33	0.157578973410137\\
34	0.158916399735448\\
35	0.160732056566942\\
36	0.161643563940769\\
37	0.164364368926586\\
38	0.165661974791265\\
39	0.167771660648864\\
40	0.169517678162322\\
41	0.170968649015173\\
42	0.172023094017198\\
43	0.172976193337988\\
44	0.174062611517435\\
45	0.17473681026671\\
46	0.175746889455155\\
47	0.177349391842904\\
48	0.177916804984633\\
49	0.180138803452433\\
50	0.180768776588813\\
};

\addplot [color=mycolor4, mark=star, mark size=1.5pt, mark options={solid, mycolor4}, forget plot]
  table[row sep=crcr]{%
1	0.618589639170852\\
2	0.373861181336824\\
3	0.284912311088637\\
4	0.240294083500818\\
5	0.223717298395481\\
6	0.174680043850606\\
7	0.165868483707827\\
8	0.165182385899257\\
9	0.102882083179722\\
10	0.0730707204710473\\
11	0.0743286314499165\\
12	0.0738940938190275\\
13	0.0741886361737198\\
14	0.0774201978178802\\
15	0.0793079720597009\\
16	0.0870736067822675\\
17	0.0925762371178167\\
18	0.0969376864446854\\
19	0.0969139760099855\\
20	0.103171250702647\\
21	0.108215555421972\\
22	0.115393664244228\\
23	0.123152864389872\\
24	0.130232820615595\\
25	0.128262781207023\\
26	0.13716945779864\\
27	0.137710359389749\\
28	0.137466719613959\\
29	0.14328893753961\\
30	0.147188768719166\\
31	0.145786979398744\\
32	0.150318623856307\\
33	0.152954122122787\\
34	0.154644997499959\\
35	0.156096715718591\\
36	0.158638818798317\\
37	0.156893080190095\\
38	0.160497726760096\\
39	0.163151943727467\\
40	0.165254941687993\\
41	0.166681666700965\\
42	0.167475403782193\\
43	0.168491731527563\\
44	0.171997781454746\\
45	0.17419403532988\\
46	0.174695307488561\\
47	0.179379186256456\\
48	0.181794797118334\\
49	0.179052467535998\\
50	0.180857850933241\\
};
\addplot [color=mycolor5, mark=diamond, mark size=1.5pt, mark options={solid, mycolor5}, forget plot]
  table[row sep=crcr]{%
1	0.618499966109694\\
2	0.373737886096262\\
3	0.284012738863235\\
4	0.225003687985915\\
5	0.191163878994616\\
6	0.178972037271361\\
7	0.195640730669385\\
8	0.197986767016056\\
9	0.199237817550365\\
10	0.199623249209613\\
11	0.199748592660204\\
12	0.199904905073718\\
13	0.199937598507671\\
14	0.199949049801451\\
15	0.19998784947183\\
16	0.199988733273918\\
17	0.199988768002029\\
18	0.199988927560884\\
19	0.199989259455013\\
20	0.199989792020422\\
21	0.199989683788683\\
22	0.199991588345755\\
23	0.19999310421258\\
24	0.199995421820543\\
25	0.199995781270173\\
26	0.199996205806492\\
27	0.199996418229242\\
28	0.199996382943967\\
29	0.199997002298205\\
30	0.199996703110982\\
31	0.199997448671347\\
32	0.199997298891791\\
33	0.199998213970143\\
34	0.199998320832974\\
35	0.199998308081133\\
36	0.199998412320141\\
37	0.199998551138417\\
38	0.199998814643202\\
39	0.19999881940616\\
40	0.199999127956214\\
41	0.199999151991303\\
42	0.199999224981078\\
43	0.199999360307343\\
44	0.199999305350886\\
45	0.199999469600652\\
46	0.199999529349299\\
47	0.199999546016866\\
48	0.199999557088954\\
49	0.19999960018044\\
50	0.199999621385493\\
};
\addplot [color=mycolor6, mark=+,  mark size=1.5pt,mark options={solid, mycolor6}, forget plot]
  table[row sep=crcr]{%
1	0.619867485732804\\
2	0.373870280292079\\
3	0.284622516613684\\
4	0.228359166019914\\
5	0.19179458789328\\
6	0.173377077348029\\
7	0.16034823732783\\
8	0.144108919682869\\
9	0.102008858314671\\
10	0.0500590101577356\\
11	0.0486758892548173\\
12	0.044115248732765\\
13	0.0405264144354487\\
14	0.0377687678080193\\
15	0.0349723656195324\\
16	0.0336095623124302\\
17	0.0313052715125948\\
18	0.0312742414455334\\
19	0.0281926915764759\\
20	0.0268317860845863\\
21	0.0265602896246356\\
22	0.0246722904338821\\
23	0.0240861507816072\\
24	0.0240933208910984\\
25	0.0227544983804929\\
26	0.0223330270610141\\
27	0.0216726994992138\\
28	0.0216576043940411\\
29	0.020405555239077\\
30	0.0193120361398142\\
31	0.0185189137452649\\
32	0.0184431360526756\\
33	0.0173965651053648\\
34	0.0170381942021612\\
35	0.0168735061085491\\
36	0.0156688065875761\\
37	0.0156678068918015\\
38	0.0156424441711125\\
39	0.0156347560773416\\
40	0.0154910379751873\\
41	0.0154455640203186\\
42	0.0142404484588016\\
43	0.0133519416133078\\
44	0.0133256459858833\\
45	0.0129075544305348\\
46	0.0125036295515441\\
47	0.0123460284240168\\
48	0.0124332392440459\\
49	0.0118301001214956\\
50	0.0114333674469342\\
};
\end{axis}

\end{tikzpicture}%